\documentclass[12pt]{amsart}

\usepackage{amssymb, amsmath, mathtools, enumerate}

\usepackage[all]{xy}

\usepackage[colorlinks=true,pagebackref,hyperindex]{hyperref}

\usepackage[utf8]{inputenc}
\usepackage[T1]{fontenc}

\usepackage{marginnote}

\usepackage{amscd}

\usepackage{tikz}

\usepackage{mathrsfs}

\usepackage{amsopn}

\usepackage{mathtools, geometry, fullpage}

\newtheorem{Theoremx}{Theorem}
 
\newtheorem{Theorem}{Theorem}[section]
\newtheorem{Lemma}[Theorem]{Lemma}
\newtheorem{Corollary}[Theorem]{Corollary}
\newtheorem{Proposition}[Theorem]{Proposition}

\newtheorem{Question}[Theorem]{Question}

\newtheorem*{Theorem*}{Theorem}

\theoremstyle{remark}
\newtheorem{Remark}[Theorem]{Remark}

\DeclareMathOperator{\height}{ht\,}

\DeclareMathOperator{\Ann}{Ann\,}

\DeclareMathOperator{\Hom}{Hom\,}
\DeclareMathOperator{\Spec}{Spec\,}

\DeclareMathOperator{\im}{im\,}
\DeclareMathOperator{\rk}{rk\,}

\DeclareMathOperator{\coker}{coker\,}

\DeclareMathOperator{\frk}{frk}
\DeclareMathOperator{\divisor}{div}

\def\p{\mathfrak{p}}
\def\q{\mathfrak{q}}
\def\m{\mathfrak{m}}
\def\a{\mathfrak{a}}
\def\e{e_{HK}}

\def\R{\mathbb{R}}
\def\Q{\mathbb{Q}}

\def\N{\mathbb{N}}
\def\O{\mathcal{O}}

\def\vaprhi{\varphi}

\newcommand{\sC}{\mathcal{C}}
\newcommand{\sD}{\mathcal{D}}
\renewcommand{\:}{\colon}
\newcommand{\blank}{\underline{\hskip 10pt}}

\title{$F$-signature and Hilbert-Kunz Multipicity: \\ a combined approach and comparison}
\author{Thomas Polstra}
\email{tmpxv3@mail.missouri.edu}
\address{Department of Mathematics, University of Missouri-Columbia, Columbia, MO 65211}
\author{Kevin Tucker}
\email{kftucker@uic.edu}
\address{Department of Mathematics\\ University of Illinois at Chicago\\Chicago\\  IL 60607}
\thanks{The second author is grateful to the NSF for partial support under Grants DMS \#1419448 and \#1602070, and for a fellowship from the Sloan Foundation. }

\begin{document}

\begin{abstract} We present a unified approach to the study of \mbox{$F$-signature}, Hilbert-Kunz multiplicity, and related limits governed by Frobenius and Cartier linear actions in positive characteristic commutative algebra.  We introduce general techniques that give vastly simplified proofs of existence, semicontinuity, and positivity. Furthermore, we give an affirmative answer to a question of Watanabe and Yoshida allowing the $F$-signature to be viewed as the infimum of relative differences in the Hilbert-Kunz multiplicites of the cofinite ideals in a local ring. 
\end{abstract}

\maketitle

\section{Introduction}
Throughout this paper, we shall assume all rings $R$ are commutative and Noetherian with prime characteristic $p > 0$. Central to the study of such rings is the use of the Frobenius or $p$-th power endomorphism $F \: R \to R$ defined by $r \mapsto r^p$ for all $r \in R$.  Following the result of Kunz \cite{KunzCharLocalRings} characterizing regularity by the flatness of Frobenius, it has long been understood that the behavior of Frobenius governs the singularities of such rings. In this article, we will be primarily concerned with two important numerical invariants  that measure the failure of flatness for the iterated Frobenius: the Hilbert-Kunz multiplicity \cite{Monsky83} and the $F$-signature \cite{SmithVandenBergh,HunekeLeuschke}.  Our aim is to revisit a number of core results about these invariants -- existence \cite{Tucker2012}, semicontinuity \cite{SmirnovUSC,PolstraSemicont}, positivity \cite{HHTams,AbLeu} -- and provide vastly simplified proofs, which in turn yield new and important results.  In particular, we confirm the suspicion of Watanabe and Yoshida allowing the \mbox{$F$-signature} to be viewed as the infimum of relative differences in the Hilbert-Kunz multiplicites of the cofinite ideals in a local ring \cite[Question~1.10]{WY}.

For the sake of simplicity in introducing Hilbert-Kunz multiplicity and $F$-signature, assume that $(R,\m,k)$ is a complete local domain of dimension $d$ and $k=k^{1/p}$ is perfect. If $I \subseteq R$ is an ideal with finite colength $\ell_R(R/I) < \infty$, the Hilbert-Kunz multiplicity along $I$ is defined by $e_{HK}(I) =\lim_{e \to \infty} \frac{1}{p^{ed}} \ell_R(R/I^{[p^e]})$ where $I^{[p^e]} = (F^e(I))$ is the expansion of $I$ along the $e$-iterated Frobenius.  Note that, unlike the case for the Hilbert-Samuel multiplicity, the function $e \mapsto \ell_R(R/I^{[p^e]})$ is far from polynomial in $p^e$ \cite{HanMonskySurprise}. The existence of Hilbert-Kunz limits was shown by Monsky \cite{Monsky83}, and it has recently been shown by Brenner that there exist irrational Hilbert-Kunz multiplicities \cite{BrennerIrratHK}.
Of particular interest is the Hilbert-Kunz multiplicity along the maximal ideal $\m$ denoted by $e_{HK}(R) = e_{HK}(\m)$, where
it is known that $e_{HK}(R) \geq 1$ with equality  if and only if $R$ is regular  \cite[Theorem~1.5]{WYHK1}. More generally, if the Hilbert-Kunz multiplicity of $R$ is sufficiently small, then $R$ is Gorenstein and strongly F-regular \cite[Proposition~2.5]{BlickleEnescu}  \cite[Corollary~3.6]{AberbachEnescuLower}.  Recently, it has been shown that the Hilbert-Kunz multiplicity determines an upper semicontinuous $\R$-valued function on ring spectra \cite{SmirnovUSC}.

Another useful perspective comes from viewing the Hilbert-Kunz function $\ell_R(R/\m^{[p^e]}) = \mu_R(R^{1/p^e})$ as measuring the minimal number of generators of the finitely generated $R$-module $R^{1/p^e}$, the ring of $p^{e}$-th roots of $R$ inside an algebraic closure of its fraction field.  As the rank $\rk_{R}(R^{1/p^e})$ of this torsion free $R$-module is $p^{ed}$, we see immediately that $R^{1/p^e}$ is free -- and hence $F^e$ is flat and $R$ is regular -- if and only if $\mu_R(R^{1/p^e}) = p^{ed}$.  Simliarly, letting $a_e=\frk(R^{1/p^e})$ denote the largest rank of a free summand  appearing in an $R$-module direct sum decomposition of $R^{1/p^e}$, we have that $R^{1/p^e}$ is free if and only if $a_e = p^{ed}$. The limit $s(R) = \lim_{e \to \infty} \frac{1}{p^{ed}} \frk(R^{1/p^e}) $ was first studied in \cite{SmithVandenBergh} and revisited in \cite{HunekeLeuschke}, where it was coined the $F$-signature and shown to exist for Gorenstein rings; existence in full generality was first shown by the second author in \cite{Tucker2012}.  Aberbach and Leuschke  \cite{AbLeu} have shown that the positivity of the $F$-signature characterizes the notion of strong $F$-regularity introduced by Hochster and Huneke \cite{HHTams} in their celebrated study of tight closure \cite{HHJams}.  Recently, it has been shown by the first author that the $F$-signature determines a lower semicontinuous $\R$-valued function on ring spectra \cite{PolstraSemicont}; an unpublished and independent proof was simultaneously found and shown to experts by the second author, and has been incorporated into this article.

At the heart of this work is an elementary argument simultaneously proving the existence of Hilbert-Kunz multiplicity and $F$-signature  (Theorem~\ref{easyproof}), derived from basic properties of the functions $\mu_R(\blank)$ and $\frk_R(\blank)$.  Building upon the philosophy introduced in \cite{Tucker2012}, we further attempt to carefully track the uniform constants controlling convergence in the proof and throughout this article.  To that end, we rely heavily on the uniform bounds for Hilbert-Kunz functions over ring spectra  (Theorem~\ref{polstrabound}) shown in the work of the first author \cite[Theorem~4.3]{PolstraSemicont}.  In particular, this also allows us to give vastly simplified proofs of the upper semicontinuity of Hilbert-Kunz multiplicity (Theorem~\ref{HKsemicont}) and lower semicontinuity of the $F$-signature  (Theorem~\ref{Fsigsemicont}).

The $F$-signature has also long been known to be closely related to the relative differences $e_{HK}(I) - e_{HK}(J)$ in Hilbert-Kunz multiplicities for ideals $I \subseteq J$ with finite colength \cite[Theorem~15]{HunekeLeuschke} (see Lemma~\ref{Iecontainscolons}).  The infimum among these differences was studied independently by Watanabe and Yoshida \cite{WY}, and the connection to $F$-signature was made by Yao \cite{Yao}.   Moreover, Watanabe and Yoshida expected and formally questioned if the infimum of the relative Hilbert-Kunz differences was equal to the $F$-signature in \cite[Question~1.10]{WY}.  We confirm their suspicions in Section \ref{Section F-signature and Minimal Relative Hilbert-Kunz Multiplicity}.

\begin{Theoremx}[Corollary \ref{WY Type result corollary}] 
\label{introWYresult}
If $(R,\m,k)$ is an F-finite local ring, then
 \[
 s(R) = \inf_{\substack{I \subseteq J  \subseteq R, \;\ell_R(R/I) < \infty  \\ I \neq J, \; \ell_R(R/J) < \infty} }\frac{e_{HK}(I) - e_{HK}(J)}{\ell_R(J/I)} = \inf_{\substack{I \subseteq R, \; \ell_R(R/I) < \infty \\  x \in R, \; ( I : x )=\m}} e_{HK}(I) - e_{HK}((I,x)).
 \]
\end{Theoremx}

Much interest in  relative differences of Hilbert-Kunz multiplicities stems from the result of Hochster and Huneke that such differences can be used to detect instances of tight closure \cite[Theorem~8.17]{HHJams}.
Together with the characterization of strong $F$-regularity in terms of the positivity of the $F$-signature \cite{AbLeu}, these two results are at the core of the use of asymptotic Frobenius techniques to measure singularities. Building off of previous work of the second author \cite{Fsigpairs1,TestIdealsFiniteMaps,TestIdealsExplicitFiniteMaps}, we present new and highly simplified proofs of these results.  In particular, the first proof of Theorem~\ref{Aberbach Leuschke} gives a readily computable lower bound for the $F$-signature of a strongly $F$-regular ring.

Using standard reduction to characteristic $p>0$ techniques, the singularities governed by Frobenius in positive characteristic commutative algebra have been closely related to those appearing in complex algebraic geometry.  This connection has motivated several important generalized settings for tight closure.  Hara and Watanabe have defined tight closure for divisor pairs \cite{HaraWatanabe} (\textit{cf.} \cite{Takagi}), and Hara and Yoshida have done the same for ideals with a real coefficient \cite{HaraYoshida}; both of these works build on the important results of Smith \cite{Smith}, Hara \cite{Hara}, and Mehta and Srinivas \cite{MehtaSrinivas}.  Previous work of the second author \cite{Fsigpairs1,Fsigpairs2} has extended the notion of $F$-signature to these settings and beyond, incorporating the notion of a Cartier subalgebra (see \cite{pinversemapssurvey}).  We work to extend all of our results as far as possible to these settings, including new and simplified proofs of existence (Theorem~\ref{Dsignatureexists}), semicontinuity (Theorems~\ref{Dsignaturesemicont}, \ref{ateefsigexists}, \ref{deltafsigexists}), and positivity (Corollaries~ \ref{deefsigpositivity}, \ref{Length Criterion for at-tight closure} (ii.), \ref{Length Criterion for Delta closure} (ii.)); many of these results have not appeared previously.  In particular, while Hilbert-Kunz theory of divisor and ideal pairs has not been introduced, we are able to to give an analogue of relative Hilbert-Kunz differences in these settings (Corollaries~\ref{Length Criterion for at-tight closure} (i.), \ref{Length Criterion for Delta closure} (i.)).  Once more, the $F$-signatures can be viewed as a minimum among the relative Hilbert-Kunz differences (Corollaries~\ref{a^t signature and WY result}, \ref{Delta signature and WY result}).

To give an overview of our methods, for each $e\in\N$ let $I_e^{\mathrm{HK}}=\m^{[p^e]}$ and $I_e^{\mathrm{F-sig}}=( r\in R \mid \psi(r^{1/p^e})\in\m\mbox{ for all }\psi\in\Hom_R(R^{1/p^e},R))$. As $k$ is perfect, straightforward computations show that $\frac{1}{p^{ed}}\mu(R^{1/p^e})=\frac{1}{p^{ed}}\ell(R/I_e^{\mathrm{HK}})$ and $\frac{1}{p^{ed}}\frk(R^{1/p^e})=\frac{1}{p^{ed}}\ell(R/I_e^{\mathrm{F-sig}})$, so that the Hilbert-Kunz multiplicity and the F-signature can be understood by studying properties of sequences of ideals $\{I_e\}_{e \in \N}$ satisfying properties similar to those of  $\{I_e^{\mathrm{HK}}\}_{e \in \N}$ and $\{I_e^{\mathrm{F-sig}}\}_{e \in \N}$.

\begin{Theoremx}[Theorems \ref{Limits exist combined}, \ref{Aberbach Leuschke type Theorem 2}] 
\label{introlimitsofsequences}
Let $(R,\m,k)$ be an F-finite local domain of dimension $d$, and consider  a sequence of ideals $\{I_e\}_{e \in \N}$ such that $\m^{[p^e]}\subseteq I_e$ for each $e$. Then the limit $\eta= \lim_{e \to \infty} \frac{1}{p^{ed}}\ell(R/I_e)$ exists provided one of the following conditions holds.
 \begin{enumerate}[(i.)]
 \item There exists $0\not=c\in R$ so that $c I_e^{[p]}\subseteq I_{e+1}$ for all $e\in\N$.
 \item There exists a nonzero $\psi\in\Hom_R(R^{1/p},R)$ so that $\psi(I_{e+1}^{1/p})\subseteq I_e$ for all $e\in\N$.
 \end{enumerate}
Moreover, in case (ii.), we have $\eta>0$ if and only if $\bigcap_{e\in\N} I_e=0$.
\end{Theoremx}

\noindent
It is easy to check that the sequences $\{I_e^{\mathrm{HK}}\}_{e \in \N}$ or $\{I_e^{\mathrm{F-sig}}\}_{e \in \N}$ satisfy both conditions (i.) and (ii.) above (with \emph{any} choice of $c$ and $\psi$).  Our results on positivity come from the stated criterion for sequences of ideals satisfying (ii.), while the relative Hilbert-Kunz statements largely come through careful tracking of uniform constants bounding the growth rates for sequences satisfying (i.).

The article is organized as follows. Section~\ref{Section Preliminaries} recalls a number of preliminary statements about rings in positive characteristic; in particular, we review properties of $F$-finite rings which will be the primary setting of this article.  We also discuss  elementary properties of the maximal free rank (see Lemma~\ref{Simple Lemma}) which will be used in the unified proof of the existence of the Hilbert-Kunz multiplicity and the $F$-signature provided in Section \ref{Section F-signature and Hilbert-Kunz Multiplicity Revisited}.  The remainder of this Section~\ref{Section F-signature and Hilbert-Kunz Multiplicity Revisited} is largely devoted to the new proofs of the semicontinuity. Section~\ref{Section Limits via p-linear and p^-1-linear maps} is the technical heart of the paper, and generalizes the methods of the previous sections for Hilbert-Kunz multiplicity and $F$-signature alone to the various sequences of ideals satisfying the conditions in Theorem~\ref{introlimitsofsequences} above.

Having dispatched with the results on existence and semicontinuity, we turn in Section~\ref{Section Positivity} to a discussion of positivity statements. In particular, two simple proofs of the positivity of the $F$-signature for strongly $F$-regular rings appear at the beginning of this section.  The relative Hilbert-Kunz criteria for tight closure along a divisor or ideal pair appear at the end of this section, together with a discussion of the positivity of the $F$-splitting ratio. 
Section~\ref{Section F-signature and Minimal Relative Hilbert-Kunz Multiplicity} is devoted to relating the F-signature with minimal relative Hilbert-Kunz differences giving a proof of Theorem \ref{introWYresult}, and we conclude in Section~\ref{Section Open Questions} with a discussion of a number of related open questions.

\section{Preliminaries}\label{Section Preliminaries}

Throughout this paper, we shall assume all rings $R$ are commutative with a unit, Noetherian,
and have prime characteristic $p > 0$. If $\p \in \Spec(R)$, we let  $k(\p) = R_\p / \p R_\p$ denote  the corresponding residue field.  A local ring is a triple $(R, \m, k)$ where $\m$ is the
unique maximal ideal of the ring $R$ and $k = k(\m) = R/\m$. 

\subsection{Maximal Free Rank} For any ring $R$ and $R$-module $M$, $\ell_R(M)$ denotes the length of $M$, and $\mu_R(M)$ denotes the minimal number of generators of $M$.  When ambiguity is unlikely, we freely omit the subscript $R$ from these and similar notations.   We define the \emph{maximal free rank} $\frk_R(M)$ of $M$ to be the maximal rank of a free $R$-module quotient of $M$.   As a surjection onto a free module necessarily admits a section, $\frk_R(M)$ is also the maximal rank of a free direct summand of $M$. Equivalently, $M$ admits a direct sum decomposition $M = R^{\oplus \frk_R(M)} \oplus N$ where $N$ has no free direct summands -- a property characterized by $\phi(N) \neq R$ for all $\phi \in \Hom_R(N, R)$.  It is immediate that $\frk_R(M) \leq \mu_R(M)$ with equality if and only if $M$ is a free $R$-module.

When $R$ is a domain, we will use $\rk_R(M)$ to denote the torsion-free rank of $M$. It is easily checked that
\begin{equation}
\label{easyrankinequalities}
\frk_R(M) \leq \rk_R(M) \leq \mu_R(M)
\end{equation}
and both inequalities are strict when $M$ is not free.  Notice that, whereas $\mu(\blank)$ is  sub-additive on short exact sequences over local rings, this is not the case for $\frk(\blank)$ and motivates the following lemma.

\begin{Lemma}\label{Simple Lemma} Let $(R,\m,k)$ be a local ring. 
\begin{enumerate}[(i.)]
\item
If $M_1$ and $M_2$ are $R$-modules, then $\frk(M_1 \oplus M_2) = \frk(M_1) + \frk(M_2)$. 
\item
If $M$ is an $R$-module with $M' \subseteq M$ a submodule and $M'' = M/M'$, then 
$$\frk(M'') \leq \frk(M)\leq \frk(M')+\mu(M'').$$
\end{enumerate}
\end{Lemma}

\begin{proof} (i.) 
If $M_1 = R^{\oplus \frk_R(M_1)} \oplus N_1$ and $M_2 = R^{\oplus \frk_R(M_2)} \oplus N_2$ where $N_1,N_2$ have no free direct summands, then $M_1 \oplus M_2 = R^{\oplus ( \frk_R(M_1) + \frk_R(M_2))} \oplus (N_1 \oplus N_2)$.  If $\phi \in \Hom_R(N_1 \oplus N_2,R)$, then $\phi(N_1),\phi(N_2) \subseteq \m$ as $N_1, N_2$ have no free direct summands.  It follows that $\phi(N_1 \oplus N_2) = \phi(N_1) + \phi(N_2) \subseteq \m$ as well, showing $N_1 \oplus N_2$ has no free direct summands as desired.

(ii.)
For the first inequality, a surjection  $M'' \to  R^{\oplus \frk(M'')}$ induces another $M \to M / M' = M'' \to R^{\oplus \frk(M'')}$, showing $\frk(M'') \leq \frk(M)$.  To show the final inequality, let $n$ be the maximal rank of a mutual free direct summand of $M$ and $M'$.  In other words, simultaneously decompose $M = R^{\oplus n} \oplus N$ and $M' = R^{\oplus n} \oplus N'$ where  $M' \subseteq M$ is given by equality on $R^{\oplus n}$ and an inclusion $N' \subseteq N$, and we have $\phi(N') \subseteq \m$ for every $\phi \: N \to R$.  Taking a surjection $\Phi \: N \to R^{\oplus \frk(N)}$, it follows $\Phi(N') \subseteq \m^{\oplus \frk(N)}$.  Thus, $\Phi$ induces a surjection $M'' = N/N' \to k^{\oplus \frk(N)}$ and hence also $M''/\m M'' \to k^{\oplus \frk(N)}$ which shows $\mu(M'') \geq \frk(N)$. This gives $\frk(M) = n + \frk(N)  \leq  \frk(M') + \mu(M'')$ as desired.
 \end{proof}

 For $X$ any topological space, recall that a function $f \: X \to \R$ is lower semicontinuous if and only if for any $\delta \in \R$ the set $\{ x \in X \, | \, f(x) > \delta \}$ is open.  Similarly, $f \: X \to \R$ is upper semicontinuous if and only if for any $\delta \in \R$ the set $\{ x \in X \, | \, f(x) < \delta \}$ is open.  
 
\begin{Lemma}
\label{easysemicont}
 If $M$ is a finitely generated $R$-module,  the function $\Spec(R) \to \R$ given by $\p \mapsto \frk_{R_\p}(M_\p)$ is lower semicontinuous.  Similarly, the function $\Spec(R) \to \R$ given by $\p \mapsto \mu_{R_\p}(M_\p)$ is upper semicontinuous. 
\end{Lemma}

\begin{proof}
For $\p \in \Spec(R)$, let $k(\p) = R_\p / \p R_\p$ denote the corresponding residue field. If $\frk_{R_{\p}}(M_\p) = n$, we can find a surjection $ M_\p \to R_\p^{\oplus n} $.  Without loss of generality, since $M$ is finitely generated, we may assume this surjection is the localization of an $R$-module homomorphism $M \to R^{\oplus n}$.  As it becomes surjective when localized at $\p$, there is some $g \in R \setminus \p$ so that $M_g \to R_g^{\oplus n} $ is surjective.  Localizing at any $\q \in \Spec(R)$ with $g \not\in\q$ yields a surjection $M_\q \to R_\q^{\oplus n}$, implying $\frk_{R_{\q}}(M_\q) \geq n = \frk_{R_{\p}}(M_\p)$ for any $\q$ in the Zariski open neighborhood $D(g) = \{ \q \in \Spec(R) \, | \, g \not\in \q \}$ of $\p$ in $\Spec(R)$. If $\delta \in \R$ and $\frk_{R_{\p}}(M_\p) > \delta $, then so also $\frk_{R_{\q}}(M_\q) > \delta$ for any $\q \in D(g)$.  This shows the function $\Spec(R) \to \R$ given by $\p \mapsto \frk_R(M_\p)$ is lower semicontinuous. That the function $\Spec(R) \to \R$ given by $\p \mapsto \mu(M_\p)$ is upper semicontinuous proceeds in an analogous manner.
\end{proof}

\subsection{F-finite rings}
\label{F-finite rings}
The Frobenius or $p$-th power endomorphism $F \: R \to R$ is defined by $r \mapsto r^p$ for all $r \in R$.  Similarly, for $e \in \mathbb{N}$, we have $F^e \: R \to R$ given by $r \mapsto r^{p^e}$.  The expansion of an ideal $I$ over $F^e$ is denoted $I^{[p^e]} = ( F^e(I) )$.  In case $R$ is a domain, we let $R^{1/p^e}$ denote the subring of $p^e$-th roots of $R$ inside a fixed algebraic closure of the fraction field of $R$.  By taking $p^e$-th roots, $R^{1/p^e}$ is abstractly isomorphic to $R$, and the Frobenius map is identified with the ring extension $R \subseteq R^{1/p^e}$.  More generally, for any $R$ and any $R$-module $M$, we will write $M^{1/p^e}$ for the $R$-module given by restriction of scalars for $F^e$; this is an exact functor on the category of $R$ modules.  We say that $R$ is $F$-finite when $R^{1/p^e}$ is a finitely generated $R$-module, which further implies that $R$ is excellent, \cite[Theorem~2.5]{Kunz}.  A complete local ring $(R, \m, k)$ is $F$-finite if and only if its residue field $k$ is $F$-finite.  When $R$ is $F$-finite and $M$ is a finitely generated $R$-module, $M^{1/p^e}$ is also a finitely generated $R$-module and 
\[
\ell_R(M^{1/p^e}) = [k^{1/p^e}:k] \cdot \ell_{R^{1/p^e}}(M^{1/p^e}) = [k^{1/p^e}:k] \cdot \ell_R(M).
\]
In particular, note that
\begin{equation}
\label{mingenstoHK}
\mu_R(M^{1/p^e}) = \ell_R(M^{1/p^e} \otimes_R k) = \ell_R((M / \m^{[p^e]}M)^{1/p^e}) = [k^{1/p^e}:k] \cdot \ell_R(M / \m^{[p^e]}M).
\end{equation}

\begin{Lemma}
\label{ranklemma}
 Suppose $(R, \m, k)$ is a complete local $F$-finite domain with coefficient field $k$ and system of parameters $x_1, \ldots, x_d$ chosen so that $R$ is module finite and generically separable over the regular subring $A = k[[ x_1, \ldots, x_d]]$.  Then there exists $0 \neq c \in A$ so that $c \cdot R^{1/p^e} \subseteq R[A^{1/p^e}] = R \otimes_A A^{1/p^e}$ for all $e \in \N$.  
\end{Lemma}

\begin{proof}
Note first that, by the Cohen-Gabber Structure Theorem, every complete local ring admits such a coefficient field and system of parameters; see \cite{CohenGabber} for an elementary proof. Let $K = \mathrm{Frac}(A) \subseteq L = \mathrm{Frac}(R)$ be the corresponding extension of fraction fields.  
The separability assumption implies $L$ and $K^{1/p^e}$ are linearly disjoint over $K$, so that $L \otimes_K K^{1/p^e} = LK^{1/p^e}$ is reduced.
Since $A^{1/p^e}$ is a free $A$-module, $R \otimes_A A^{1/p^e}$ is a free $R$-module and injects into its localization $L \otimes_A A^{1/p^e} = L \otimes_K K^{1/p^e}$. Thus $R \otimes_A A^{1/p^e}$ is also reduced, which gives the equality $R \otimes_A A^{1/p^e} = R[A^{1/p^e}]$.  If $e_1, \ldots, e_s \in R$ generate $R$ as an $A$-module, it is easy to see $c = \left( \mathrm{det}\left( \mathrm{tr}(e_i e_j)\right)\right)^2$ satisfies $c \cdot R^{1/p^e} \subseteq R[A^{1/p^e}]$ where $\mathrm{tr}$ is the trace map; see \cite[Lemma~4.3]{HochsterHunekeTextExponents} for a proof and \cite[Section~6]{HHJams} for further discussion.
\end{proof}

The Corollaries below follow immediately; see \cite{Kunz} for details.

\begin{Corollary}
 If $(R, \m, k)$ is an $F$-finite local domain of dimension $d$, then $\rk_R(R^{1/p^e}) = [k^{1/p^e}:k] \cdot p^{ed}$.  
\end{Corollary}

\begin{Corollary}
\label{gamma}
 If $R$ is $F$-finite, for any two prime ideals $\p \subseteq \q$ of $R$ it follows  $$[k(\p)^{1/p^e}:k(\p)] = [k(\q)^{1/p^e}:k(\q)]  \cdot p^{e \dim(R_\q / \p R_\q)}.$$  If additionally $R$ is locally equidimensional, we have $$[k(\p)^{1/p^e}:k(\p)] \cdot p^{e \dim(R_\p)} = [k(\q)^{1/p^e}:k(\q)]  \cdot p^{e \dim(R_\q)}$$ and the function $\Spec(R) \to \N$ given by $\p \mapsto [k(\p)^{1/p^e}:k(\p)] \cdot p^{e \dim(R_\p)}$ is constant on the connected components of $\Spec(R)$.
\end{Corollary}

\begin{Corollary}
\label{shortexactsequences}
 If $R$ is a locally equidimensional $F$-finite reduced ring with connected spectrum and $p^{\gamma} = [k(\p)^{1/p}:k(\p)] \cdot p^{\dim(R_\p)}$ for all $\p \in \Spec(R)$, there exist short exact sequences
\begin{equation}
\label{plinearinclusion}
\xymatrix{ 
0 \ar[r] & R^{\oplus p^\gamma} \ar[r] & R^{1/p} \ar[r] & M \ar[r] & 0}
\end{equation}
\begin{equation}
\label{pinverselinearinclusion}
\xymatrix{
0 \ar[r] &   R^{1/p} \ar[r] & R^{\oplus p^\gamma} \ar[r] & N \ar[r] & 0
}
\end{equation}
so that $\dim(M) < d$ and $\dim(N) < d$.
\end{Corollary}

\section{F-signature and Hilbert-Kunz Multiplicity Revisited}\label{Section F-signature and Hilbert-Kunz Multiplicity Revisited}

The first goal of this section is to provide an elementary proof that simultaneously shows the existence of the $F$-signature and Hilbert-Kunz multiplicity for an $F$-finite local domain $(R, \m, k)$.  At the heart of all known proofs of existence is a basic observation on the growth rate of a Hilbert-Kunz function. 

\begin{Lemma}
\label{Well known bound}
\cite[Lemma 1.1]{Monsky83} Let $(R,\m,k)$ be a local ring and $M$  a finitely generated $R$-module. Then there exists a positive constant $C(M, \m) \in\R$ so that 
\[
\ell(M/\m^{[p^e]}M)\leq C(M,\m) \cdot p^{e \dim(M)}
\]
for each $e\in\N$.
\end{Lemma}

\begin{proof} If $\m$ is generated by $t$ elements, then $\m^{tp^e} \subseteq \m^{[p^e]}$ for each $e\in\N$. Thus $\ell(M/\m^{[p^e]}M)\leq \ell(M/\m^{tp^e}M)$, which is eventually polynomial of degree $\dim(M)$ in the entry $p^e$.\end{proof}

\begin{Theorem}
\label{easyproof}
\cite[Theorem 1.8]{Monsky83}
\cite[Theorem 4.9]{Tucker2012}
 Suppose that $(R, \m, k)$ is an $F$-finite local domain and $\dim(R) = d$.  Then the limits
 \[
 e_{HK}(R) = \lim_{e \to \infty} \frac{\mu_R(R^{1/p^e})}{[k^{1/p^e}:k] \cdot p^{ed}}  \qquad s(R) = \lim_{e \to \infty} \frac{\frk_R(R^{1/p^e})}{[k^{1/p^e}:k] \cdot p^{ed}} 
 \]
 exist.
\end{Theorem}

\begin{proof}
Let $p^\gamma = [k^{1/p}:k] \cdot p^d = \rk_R(R^{1/p^e})$ and $\nu(\blank)$ denote either $\mu(\blank)$ or $\frk(\blank)$. Set $\eta_e = \frac{1}{p^{e\gamma}} \nu_R(R^{1/p^e})$, and observe using Lemma~\ref{Well known bound} with \eqref{easyrankinequalities} and \eqref{mingenstoHK} that the sequence $\{ \eta_e \}_{e \in \N}$ is bounded above.  Put $\eta^+ = \limsup_{e \to \infty} \eta_e$ and $\eta^- = \liminf_{e \to \infty} \eta_e$.
 
 As in Lemma~\ref{shortexactsequences} (2), fix a short exact sequence
 \[
 \xymatrix{0 \ar[r] & R^{\oplus p^\gamma} \ar[r] & R^{1/p} \ar[r] & M \ar[r] & 0}
 \]
 where $\dim(M) < d$.  Applying the exact functors $(\blank)^{1/p^e}$ gives the short exact sequences
 \[
  \xymatrix{0 \ar[r] & (R^{1/p^e})^{\oplus p^\gamma} \ar[r] & R^{1/p^{e+1}} \ar[r] & M^{1/p^e} \ar[r] & 0}
 \]
 and Lemma~\ref{Simple Lemma} implies
 \[
 \nu(R^{1/p^{e+1}}) \leq p^\gamma \cdot \nu(R^{1/p^e}) + \mu(M^{1/p^e})
 \]
for each $e \in \mathbb{N}$.  By Lemma~\ref{Well known bound} and using \eqref{mingenstoHK}, there exists a positive constant $C(M,\m) \in \R$ with $\mu(M^{1/p^e}) \leq C(M,\m) p^{e \dim(M)} \leq C(M, \m) p^{e(d-1)}$. Dividing through by $p^{(e+1)d}$ and setting $C = \frac{C(M,\m)}{p^d} \in \R$ yields 
$
\eta_{e+1} \leq \eta_e + \frac{C}{p^e}.
$
Iterating this inequality gives
\[
\eta_{e + e'} \leq \eta_{e} + \frac{C}{p^e} \left( 1 + \frac{1}{p} + \cdots + \frac{1}{p^{e'-1}} \right) \leq \eta_e + \frac{2C}{p^e}
\]
for all $e, e' \in \N$.  For each $e$, taking $\limsup_{e' \to \infty}$ then gives $\eta^+ \leq \eta_e + \frac{2C}{p^e}$.  Now taking $\liminf_{e \to \infty}$ gives $\eta^+ \leq \eta^-$ as desired.
\end{proof}

Our next goal is to give a direct proof of the upper semicontinuity of Hilbert-Kunz multiplicity.  The key observation is that the constant appearing in Lemma~\ref{Well known bound} can be taken uniformly on ring spectra.

\begin{Theorem}
\label{polstrabound}
\cite[Theorem 4.3]{PolstraSemicont}
For any $F$-finite ring $R$ and finitely generated module $M$, there exists a positive constant $C(M)$ such that for all $\p \in \Spec(R)$ 
$$\ell_{R_\p}(M_\p / \p^{[p^e]}M_\p) \leq C(M) p^{e \dim(M_\p)}.$$ 
\end{Theorem}

\noindent
From the proof of Theorem~\ref{easyproof} above, which is valid without change for a local equidimensional reduced ring, it follows that there is a positive constant $C \in \R$ such that
\[
e_{HK}(R_\p) \leq \frac{\mu_{R_\p}(R_\p^{1/p^e})}{[k(\p)^{1/p^e}:k(\p)] \cdot p^{e \dim(R_\p)}} + \frac{2C}{p^e} \qquad s(R_\p) \leq \frac{\frk_{R_\p}(R_\p^{1/p^e})}{[k(\p)^{1/p^e}:k(\p)] \cdot p^{e \dim(R_\p)}} + \frac{2C}{p^e}
\]
 for all $\p \in \Spec(R)$ and all $e \in \N$.  
 
\begin{Theorem}\cite[Main Result]{SmirnovUSC}
\label{HKsemicont}
For any $F$-finite locally equidimensional reduced ring $R$, the function $\e:\Spec(R) \to \R$ given by $\p \mapsto e_{HK}(R_\p)$ is upper-semicontinuous.
\end{Theorem}

\begin{proof}
Restricting to a connected component, we may assume without loss of generality that $\Spec(R)$ is connected and hence $p^\gamma = [k(\p)^{1/p}:k(\p)] \cdot p^{\dim(R_\p)}$ is constant for all $\p \in \Spec(R)$ by Corollary \ref{gamma}.  If $\delta \in \R$ and $e_{HK}(R_{\q}) < \delta$ for some $\q \in \Spec(R)$, then $$\frac{\mu_{R_\q}(R_\q^{1/p^e})}{[k(\q)^{1/p^e}:k(\q)] \cdot p^{e \dim(R_\q)}} < \delta - \epsilon$$ for some $0 < \epsilon \ll 1$ and some $e \in \N$ with $\frac{2C}{p^e} < \epsilon$.  By Lemma~\ref{easysemicont}, the same holds true for all $\p$ in a neighborhood of $\q$.  This yields
\[
e_{HK}(R_\p) \leq  \frac{\mu_{R_\p}(R_\p^{1/p^e})}{[k(\p)^{1/p^e}:k(\p)] \cdot p^{e \dim(R_\p)}} + \frac{2C}{p^e} < \delta - \epsilon + \frac{2C}{p^e} < \delta
\]
as desired and completes the proof.
\end{proof}

To get a similar argument for  lower semicontinuity of $F$-signature, we need to reverse the estimates arising in the proof of existence.  To that end, we first record the following elementary lemma.

\begin{Lemma}\label{Sequence Lemma} Let $p$ a prime number, $d \in \N$, and  $\{\lambda_e\}_{e \in \N}$ be a sequence of real numbers so that $\{\frac{1}{p^{ed}}\lambda_e\}_{e \in \N}$ is a bounded.
\begin{enumerate}[(i.)]
\item
If there exists a positive constant $C \in \R$ so that $\frac{1}{p^{(e+1)d}} \lambda_{e+1} \leq \frac{1}{p^{ed}} \lambda_e + \frac{C}{p^e}$ for all $ e \in \N$, then the limit $\lambda = \lim_{e \to \infty} \frac{1}{p^{ed}} \lambda_e$ exists and $\lambda - \frac{1}{p^{ed}} \lambda_e \leq \frac{2C}{p^e}$ for all $e \in \N$.
\item
If there exists a positive constant $C \in \R$ so that $\frac{1}{p^{ed}} \lambda_e \leq  \frac{1}{p^{(e+1)d}} \lambda_{e+1} + \frac{C}{p^e}$ for all $e \in \N$, then the limit $\lambda = \lim_{e \to \infty} \frac{1}{p^{ed}} \lambda_e$ exists and $\frac{1}{p^{ed}} \lambda_e - \lambda \leq \frac{2C}{p^e}$ for all $e \in \N$.
\item
If there exists a positive constant $C \in \R$ so that $| \frac{1}{p^{(e+1)d}} \lambda_{e+1} - \frac{1}{p^{ed}} \lambda_e    | \leq \frac{C}{p^e}$ for all $e \in \N$, then the limit $\lambda = \lim_{e \to \infty} \frac{1}{p^{ed}} \lambda_e$ exists and $|\frac{1}{p^{ed}} \lambda_e - \lambda | \leq \frac{2C}{p^e}$ for all $e \in \N$.  In particular, $\lambda_e=\lambda p^{ed}+O(p^{e(d-1)})$.
\end{enumerate}
\end{Lemma}

\begin{proof} Let $\lambda^+ = \limsup_{e \to \infty} \frac{1}{p^{ed}} \lambda_e$ and $\lambda^- = \liminf_{e \to \infty} \frac{1}{p^{ed}} \lambda_e$, which are finite as $\{\frac{1}{p^{ed}}\lambda_e\}_{e \in \N}$ is bounded.  
Iterating the inequality in (i.) yields 
\[
\frac{1}{p^{(e+e')d}} \lambda_{e + e'} \leq \frac{1}{p^{ed}}\lambda_{e} + \frac{C}{p^e} \left( 1 + \frac{1}{p} + \cdots + \frac{1}{p^{e'-1}} \right) \leq \frac{1}{p^{ed}}\lambda_{e} + \frac{2C}{p^e}
\]
for all $e, e' \in \N$.  Taking $\limsup_{e' \to \infty}$ for each fixed $e \in \N$ gives $\lambda^+ \leq \frac{1}{p^{ed}}\lambda_{e} + \frac{2C}{p^e}$, and then applying $\liminf_{e \to \infty}$ gives $\lambda^+ \leq \lambda^-$ so that $\lambda = \lim_{e \to \infty} \frac{1}{p^{ed}}\lambda_{e} $ exists and $\lambda - \frac{1}{p^{ed}} \lambda_e \leq \frac{2C}{p^e}$ for all $e \in \N$.  
Similarly, iterating the inequality in (ii.) yields 
\[
\frac{1}{p^{ed}}\lambda_{e} \leq  \frac{1}{p^{(e+e')d}} \lambda_{e + e'}+ \frac{C}{p^e} \left( 1 + \frac{1}{p} + \cdots + \frac{1}{p^{e'-1}} \right) \leq  \frac{1}{p^{(e+e')d}} \lambda_{e + e'} + \frac{2C}{p^e}
\]
for all $e, e' \in \N$.  Taking $\liminf_{e' \to \infty}$ for each fixed $e \in \N$ gives $\frac{1}{p^{ed}}\lambda_{e} \leq \lambda^- + \frac{2C}{p^e}$, and then applying $\limsup_{e \to \infty}$ gives $\lambda^+ \leq \lambda^-$ so that $\lambda = \lim_{e \to \infty} \frac{1}{p^{ed}}\lambda_{e} $ exists and $ \frac{1}{p^{ed}} \lambda_e - \lambda \leq \frac{2C}{p^e}$ for all $e \in \N$.  
The final statement (iii.) follows immediately from a combination of (i.) and (ii.).
\end{proof}

\begin{Theorem}
\label{Uniform Convergence Theorem}
 If $R$ is a locally equidimensional reduced $F$-finite ring of dimension $d$, there is a positive constant $C \in \R$ so that
 \[
 \left| \frac{\mu_{R_\p}(R^{1/p^e}_\p)}{[k(\p)^{1/p^e}:k(\p)] \cdot p^{e \dim R_\p}} - e_{HK}(R_\p) \right| \leq \frac{C}{p^e}
 \qquad   
 \left| \frac{\frk_{R_\p}(R_\p^{1/p^e})}{[k(\p)^{1/p^e}:k(\p)] \cdot p^{e \dim R_\p}} - s(R_\p) \right| \leq \frac{C}{p^e}
 \]
 for all $e \in \N$ and $\p \in \Spec(R)$.
\end{Theorem}

\begin{proof}Restricting to a connected component, we may assume without loss of generality that $\Spec(R)$ is connected and hence $p^\gamma = [k(\p)^{1/p}:k(\p)] \cdot p^{\dim(R_\p)}$ is constant for all $\p \in \Spec(R)$ by Corollary \ref{gamma}.  

The proofs of Theorems \ref{easyproof} and \ref{HKsemicont} above relied on a short exact sequence as in Lemma~\ref{shortexactsequences} (2) and produced inequalities as in Lemma~\ref{Sequence Lemma} (i.).  Repeating those same arguments on a short exact sequence as in Lemma~\ref{shortexactsequences} (3) yields inequalities as in Lemma~\ref{Sequence Lemma} (ii.), which combine to give the desired result.
\end{proof}

\begin{Corollary}
 If $(R,\m,k)$ is an $F$-finite  local domain of dimension $d$,
  \[
 \frac{\mu_R(R^{1/p^e})}{ [k:k^{p^e}]} = e_{HK}(R) p^{ed} + O(p^{e(d-1)})
 \qquad \qquad
\frac{ \frk_R(R^{1/p^e})}{[k:k^{p^e}]} = s(R) p^{ed} + O(p^{e(d-1)}).
 \]
\end{Corollary}

\begin{Theorem}
\cite[Theorem 5.7]{PolstraSemicont}
\label{Fsigsemicont}
 For any $F$-finite locally equidimensional reduced ring $R$, the function $s:\Spec(R) \to \R$ given by $\p \mapsto s(R_\p)$ is lower-semicontinuous. 
\end{Theorem}

\begin{proof}
From Theorem \ref{Uniform Convergence Theorem}, there is a $C \in \R$ such that
\[
 \frac{\frk_{R_\p}(R_\p^{1/p^e})}{[k(\p)^{1/p^e}:k(\p)] \cdot p^{e \dim(R_\p)}} \leq s(R_\p) + \frac{C}{p^e}
\]
 for all $\p \in \Spec(R)$ and all $e \in \N$. The argument in Theorem \ref{HKsemicont} immediately gives the desired result.  Alternatively, using the whole of Theorem \ref{Uniform Convergence Theorem}, both semicontinuity statements follow from the fact that the uniform limit of upper (\cite[Corollary~3.4]{Kunz}) or lower (\cite[Corollary~2.5]{EnescuYaoLSC}) semicontinuous functions is upper or lower semicontinuous, respectively.
\end{proof}

\section{Limits via Frobenius and Cartier linear maps}\label{Section Limits via p-linear and p^-1-linear maps}

\subsection{Background}
Suppose that $M$ is an $R$-module.  Recall that an additive map $\phi \in \Hom_{\mathbb Z}(M,M)$ is said to be Frobenius linear or \emph{$p$-linear} if $\phi(rm) = r^p \phi(m)$ for all $r \in R$ and $m \in M$.  Similarly, we say that $\phi \in \Hom_{\mathbb Z}(M,M)$ is $p^e$-linear if $\phi(rm) = r^{p^e} m$ for all $r \in R$ and $m \in M$.  The set of all $p^e$-linear maps on $M$ is denoted $\mathcal{F}_e(M)$ and is naturally both a left and right $R$-module readily identified with $\Hom_R(M, M^{1/p^e})$.  The composition of a $p^{e_1}$-linear map $\phi_1$ and a $p^{e_2}$-linear map $\phi_2$ yields a $p^{e_1 + e_2}$-linear map $\phi_1 \circ \phi_2$, so that the ring of Frobenius linear operators $\mathcal{F}(M) = \bigoplus_{e \geq 0} \mathcal{F}_e(M)$ on $M$ forms a non-commutative $\N$-graded ring with $R \subseteq \mathcal{F}_0(M) = \Hom_R(M,M)$.  
In the case $M = R$, we have that  every $p^e$-linear map $\phi \in \Hom_R(R, R^{1/p^e}) = R^{1/p^e}$ is a post-multiple of the Frobenius $F^e$ by some $c^{1/p^e} \in R$. 

Dual to $p$-linear maps are the Cartier linear or  \emph{$p^{-1}$-linear maps}; an additive map $\phi \in \Hom_{\mathbb Z}(M,M)$ is said to be $p^{-1}$-linear if $\phi(r^p m) = r \phi(m)$ for all $r \in R$ and $m \in M$.  Similarly, we say that $\phi \in \Hom_{\mathbb Z}(M,M)$ is $p^{-e}$-linear if $\phi(r^{p^e}m) = r m$ for all $r \in R$ and $m \in M$.  The set of all $p^{-e}$-linear maps on $M$ is denoted $\mathcal{C}_e(M)$ and is naturally both a left and right $R$-module readily identified with $\Hom_R(M^{1/p^e},M)$.  The composition of a $p^{-e_1}$-linear map $\phi_1$ and a $p^{-e_2}$-linear map $\phi_2$ yields a $p^{-(e_1 + e_2)}$-linear map $\phi_1 \circ \phi_2$, so that the ring of Cartier linear operators $\mathcal{C}(M) = \bigoplus_{e \geq 0} \mathcal{C}_e(M)$ on $M$ forms a non-commutative $\N$-graded ring with $R \subseteq \mathcal{C}_0(M) = \Hom_R(M,M)$.  

If we have $\phi_1 \in \Hom_R(M^{1/p^{e_1}},M)$ and $\phi_2 \in \Hom_R(M^{1/p^{e_2}},M)$ we write $\phi_1 \cdot \phi_2$ for the composition $\phi_1 \circ (\phi_2)^{1/p^{e_1}}$, which coincides with their product when viewed as elements of $\mathcal{C}(M)$.   In particular, given $\phi \in \Hom_R(M^{1/p}, M)$ we write $\phi^e \in \Hom_R(M^{1/p^e},M)$ for the corresponding $e$-th iterate.
If $\psi\in\Hom_R(M^{1/p^e},M)$ we let $\psi(r^{1/p^e} \cdot \blank)$ denote the $R$-linear map
$$\xymatrix{M^{1/p^e} \ar[r]^{r^{1/p^e}} & M^{1/p^e} \ar[r]^{\psi} & R}$$  given by pre-multiplying with $r^{1/p^e} \in R^{1/p^e}$. 

\begin{Lemma}\label{Map Lemma} Let $R$ be an F-finite domain of dimension $d$ with $\rk_R(R^{1/p}) = p^ \gamma$. 
\begin{enumerate}[(i.)]
\item
If $0 \neq c \in R$, then there exists a short exact sequence
\[
\xymatrix{
0 \ar[r] & R^{\oplus p^\gamma} \ar[r]^{\Phi} & R^{1/p} \ar[r] & M \ar[r] & 0
}
\]
so that $\dim(M) < R$ and $\Phi(R^{\oplus p^\gamma})\subseteq (cR)^{1/p}$.
\item
Let $0 \neq \psi\in\Hom_R(R^{1/p},R)$ be a nonzero map. 
There exists a short exact sequence
\[
\xymatrix{
0  \ar[r] & R^{1/p} \ar[r]^{\Psi} & R^{\oplus p^\gamma} \ar[r] & N \ar[r]& 0
}
\]
so that every component function of $\Psi$ is a pre-multiple of $\psi$.  In other words,
there exists $r_1,...,r_{p^\gamma}\in R$ so that $\Psi=(\psi(r_1^{1/p}\cdot \blank),..., \psi(r^{1/p}_{p^{\gamma}}\cdot \blank))$.
\end{enumerate}
\end{Lemma}

\begin{proof} For (i.), start with any inclusion $R^{\oplus p^\gamma} \subseteq R^{1/p}$ with a torsion quotient as in \eqref{plinearinclusion} and post-multiply by $c^{1/p}$ on $R^{1/p}$.  For (ii.), note that $\rk_{R^{1/p}}(\Hom_R(R^{1/p},R)) = 1$ so that there is some non-zero $z \in R$ with $z^{1/p} \cdot \Hom_R(R^{1/p},R) \subseteq \psi \cdot R^{1/p}$.  Start with an inclusion $R^{1/p} \subseteq R^{\oplus p^\gamma}$ with torsion quotient as in \eqref{pinverselinearinclusion} and pre-multiply by $z^{1/p}$ on $R^{1/p}$ to achieve the desired sequence.
\end{proof}

If $R$ is an $F$-finite normal domain and $D$ is a (Weil) divisor on $X = \Spec(R)$, we use $R(D)$ to denote $\Gamma(X, \O_X(D))$.  There is a well known correspondence between $p^{-e}$-linear maps and certain effective $\Q$-divisors.  Fixing a canonical divisor $K_R$, standard duality arguments for finite extensions show that $\Hom_R(R^{1/p^e},R) = (R((1-p^e)K_R))^{1/p^e}$ for each $e \in \N$.  Thus, to each $0 \neq \phi \in \Hom_R(R^{1/p^e},R)$ we can associate a divisor $D_\phi$ so that $D_\phi \sim_\mathbb{Z} (1-p^e)K_R$, and we set $\Delta_\phi = \frac{1}{p^e - 1} D_\phi$.  If $\phi, \psi \in \Hom_R(R^{1/p^e},R)$, then $\Delta_\phi \geq \Delta_\psi$ if and only if $\phi ( \blank ) = \psi( r^{1/p^e} \cdot \blank)$ for some $0 \neq r \in R$, in which case $\Delta_\phi = \Delta_\psi + \frac{1}{p^e - 1} \mathrm{div}_R(r)$.  Moreover, if $\phi_1 \in \Hom_R(R^{1/p^{e_1}},R)$ and $\phi_2 \in \Hom_R(R^{1/p^{e_2}},R)$ with $\phi = \phi_1 \cdot \phi_2 = \phi_1 \circ (\phi_2)^{1/p^{e_1}}$, then $\Delta_\phi = \frac{1}{p^{e_1 + e_2} }( (p^{e_1} - 1) \Delta_{\phi_1}  +p^{e_1}(p^{e_2} - 1) \Delta_{\phi_2} )$.  In particular, $\Delta_{\phi} = \Delta_{\phi^n}$ for all $e,n \in \N$ and all $\phi \in \Hom_R(R^{1/p^e},R)$.  See \cite[Section 4]{TestIdealSurvey} or \cite[Section 2]{TestIdealsFiniteMaps} for further details.

\begin{Lemma}
\label{makeitaniterateofafixedmap}
Let $R$ by an $F$-finite domain and $0 \neq \psi \in \Hom_R(R^{1/p},R)$.  There exists an element $0 \neq z \in R$ so that, for any $e \in \N$ and any $\phi \in \Hom_R(R^{1/p^e},R)$ the map $z \cdot \phi(\blank) = \phi(z \cdot \blank) = \phi( (z^{p^e})^{1/p^e} \cdot \blank) = \psi^e( r^{1/p^e} \cdot \blank)$ for some $r^{1/p^e} \in R^{1/p^e}$.  In other words, $z\phi$ is a pre-multiple of $\psi^e$.
\end{Lemma}

{\renewcommand{\c}{\mathfrak{c}}
\begin{proof}
Let $\overline{R}$ be the normalization of $R$ in its fraction field, and take $0 \neq c \in R$ inside the conductor ideal $\c = \Ann_R(\overline{R}/R)$.  One can show \cite[Exercise 6.14]{pinversemapssurvey} that every $p^{-e}$-linear map $\phi \: R^{1/p^e} \to R$ extends uniquely to a $p^{-e}$-linear map on $\overline{R}$ compatible with $\c$, which we denoted here by $\overline{\phi}\: \overline{R}^{1/p^e} \to \overline{R}$.  Let $0 \neq x \in R$ be such that $\mathrm{div}_{\overline{R}}(x) \geq \Delta_{\overline{\psi}}$ and put $z = cx$.  Then $\Delta_{x\overline\phi} = \Delta_{\overline\phi} + \frac{p^e}{p^e -1} \mathrm{div}_{\overline{R}}(x) \geq \Delta_{\overline{\psi}} = \Delta_{\overline{\psi^e}}$ so that $\overline{\phi}(x \cdot \blank) = \psi^e(\overline{r}^{1/p^e} \cdot \blank)$ for some $\overline{r} \in \overline{R}$.  Letting $r^{1/p^e} = c \overline{r}^{1/p^e} \in R^{1/p^e}$ and restricting back to $R$, this gives $\phi(z \cdot \blank) = \psi^e(r^{1/p^e} \cdot \blank)$ as desired.
\end{proof}}

Recall that a Cartier subalgebra on $R$ is a subring $\sD \subseteq \sC(R)$ of the ring of Cartier linear operators on $R$.  Cartier subalgebras can be seen as a natural generalization of a number of commonly studied settings in positive characteristic commutative algebra.  For instance, if $R$ is an $F$-finite domain and $0 \neq \a \subseteq R$ is an ideal and $t \in \R_{\geq 0}$, the Cartier subalgebra
$\sC^{\a^t} = \bigoplus_{e \geq 0} \sC^{\a^t}_e$ where
\[
\begin{array}{rcl}
\sC^{\a^t}_e &=& \a^{\lceil t(p^e - 1)\rceil / p^e}\Hom_R(R^{1/p^e},R) \\
&=& \{ \phi ( x^{1/p^e} \cdot \blank) \mid x \in \a^{\lceil t(p^e - 1)\rceil / p^e} \mbox{ and } \phi \in \Hom_R(R^{1/p^e},R) \}
\end{array}
\]
recovers the framework of \cite{HaraYoshida}.  Similarly, if $R$ is an $F$-finite normal domain and $\Delta$ is an effective $\Q$-divisor on $\Spec(R)$, the 
Cartier subalgebra $\sC^{(R,\Delta)} = \bigoplus_{e \geq 0} \sC^{(R,\Delta)}_e$ where
\[
\begin{array}{rcl}
\sC^{(R,\Delta)}_e &=& \{ \phi \in \Hom_R(R^{1/p^e},R) \mid \Delta_\phi \geq \Delta\} \\ &=& \im\left(  \Hom_R(R(\lceil (p^e - 1) \Delta \rceil)^{1/p^e},R) \to \Hom_R(R^{1/p^e},R)\right)
\end{array}
\]
recovers the setting of \cite{HaraWatanabe,Takagi}.  For more information on Cartier subalgebras, see \cite{pinversemapssurvey}.  
We will mainly be interested in the generalization of $F$-signature to a Cartier subalgebra $\sD$ introduced in \cite{Fsigpairs1,Fsigpairs2}.

\subsection{Existence and semicontinuity}

\begin{Theorem}
\label{MainExistenceTheoremIdealSequences}
 Let $(R, \m, k)$ be an $F$-finite local domain of dimension $d$, and $\{ I_e \}_{e \in \N}$ a sequence of ideals such that $\m^{[p^e]} \subseteq I_e$ for all $e \in \N$.
 \begin{enumerate}[(i.)]
 \item
 If there exists $0 \neq c \in R$ so that $c I_e^{[p]} \subseteq I_{e + 1}$ for all $e \in \N$, then $\eta = \lim_{e \to \infty} \frac{1}{p^{ed}} \ell_R(R/I_e)$ exists.  Moreover, there exists a positive constant $C(c)$ depending only on $c \in \R$ with $\eta - \frac{1}{p^{ed}} \ell_R(R/I_e) \leq \frac{C(c)}{p^e}$ for all $e \in \N$.
 \item
  If there exists a non-zero $R$-linear map $\psi \: R^{1/p} \to R$ so that $\psi(I_{e+1}^{1/p}) \subseteq I_e$ for all $e \in \N$, then $\eta = \lim_{e \to \infty} \frac{1}{p^{ed}} \ell_R(R/I_e)$ exists.  Moreover, there exists a positive constant $C(\psi) \in \R$  depending only on $\psi$ such that $\frac{1}{p^{ed}} \ell_R(R/I_e) - \eta \leq \frac{C(\psi)}{p^e}$ for all $ e \in \N$.
 
\end{enumerate}
\end{Theorem}

\begin{Remark}
Note that the conditions on ideal sequences $\{ I_e \}_{e \in \N}$ in Theorem~\ref{MainExistenceTheoremIdealSequences} (i.) and (ii.) are far more symmetric when phrased in terms of $p$-linear and $p^{-1}$-linear maps.  In (i.), the requirement is simply that there exists a $p$-linear map $0 \neq \phi \in \mathcal{F}_1(R)$ on $R$ such that $\phi(I_e) \subseteq I_{e+1}$ for each $e\in\N$.  Similarly, for (ii.) the requirement is that there exists a $p^{-1}$-linear map $0 \neq \phi \in \mathcal{C}_1(R)$ such that $\phi(I_{e+1}) \subseteq I_{e}$ for each $e\in\N$.
\end{Remark}

\begin{proof}[Proof of Theorem~\ref{MainExistenceTheoremIdealSequences}]
For (i.), put $\rk_R(R^{1/p}) =[k^{1/p}:k] \cdot p^{d} = p^\gamma$ and consider a short exact sequence
\begin{equation}
\label{firstshortexactinmainexistence}
\xymatrix{
0 \ar[r] & R^{\oplus p^\gamma} \ar[r]^{\Phi} &  R^{1/p} \ar[r] & M \ar[r] & 0
}
\end{equation}
as in Lemma~\ref{Map Lemma} (i.) with $\dim(M)< \dim(R)$ and $ \Phi(R^{\oplus p^\gamma}) \subseteq (Rc)^{1/p}$. Note that the condition $cI^{[p^e]} \subseteq I_{e+1}$ for all $e \in \N$ can be restated as $I_e (Rc)^{1/p} \subseteq (I_{e+1})^{1/p}$, so that $\Phi(I_e^{\oplus p^\gamma}) = \Phi( I_e (R^{\oplus p^\gamma})) \subseteq I_e (Rc)^{1/p} \subseteq (I_{e+1})^{1/p}$.  In particular, $\Phi$ induces a quotient map
\[
\overline{\Phi} \: (R/I_e)^{\oplus p^\gamma} \to (R/I_{e+1})^{1/p}
\] 
and it follows that 
\[
[k^{1/p}:k] \ell_R(R/I_{e+1}) = \ell_R((R/I_{e+1})^{1/p}) \leq p^\gamma \ell_R(R/I_e) + \ell_R(\coker(\overline{\Phi})).  
\] 
Since $\m^{[p^{e+1}]} \subseteq I_{e+1}$ and $\coker(\overline{\Phi})$ is a quotient of $(R/I_{e+1})^{1/p}$, we have $\m^{[p^e]} \subseteq \Ann_R(\coker(\overline{\Phi}))$.  But $\coker(\overline{\Phi})$ is also a quotient of $\coker(\Phi) = M$ and thus $M / \m^{[p^e]}M$, so that $\ell_R(\coker(\overline{\Phi})) \leq \ell_R(M/\m^{[p^e]}M) \leq C(M) p^{e(d-1)}$ by Lemma~\ref{Well known bound}.  Dividing through by $[k^{1/p}:k]p^{(e+1)d} = p^{\gamma + ed}$ yields
\[
\frac{1}{p^{(e+1)d}} \ell_R(R/I_{e+1}) \leq \frac{1}{p^{ed}}\ell_R(R/I_e) + \frac{C(M,\m) / p^\gamma}{p^e}
\]
and the result now follows from Lemma~\ref{Sequence Lemma} with $C(c) = 2C(M, \m) / p^\gamma$ independent of the sequence of ideals $\{I_e\}_{e \in \N}$ satisfying the condition in (i.).  

Similarly for (ii.), consider a short exact sequence
\[
\xymatrix{
0 \ar[r] &  R^{1/p}  \ar[r]^{\Psi} &  R^{\oplus p^\gamma}\ar[r] & N \ar[r] & 0
}
\]
as in Lemma~\ref{Map Lemma} (ii.)  with $\dim(N)< \dim(R)$ so that every component function of $\Psi$ is a pre-multiple of $\psi$. It follows that $\Psi((I_{e+1})^{1/p}) \subseteq I_e^{\oplus p^\gamma}$, so that $\Psi$ induces a quotient map
\[
\overline{\Psi} \: (R/I_{e+1})^{1/p} \to (R/I_e)^{\oplus p^\gamma}
\] 
and thus
\[
\begin{array}{rcl}
  p^\gamma \ell_R(R/I_e) 
&\leq&
 \ell_R((R/I_{e+1})^{1/p})+ \ell_R(\coker(\overline{\Psi})) \\
  &=& 
 [k^{1/p}:k] \ell_R(R/I_{e+1}) + \ell_R(\coker(\overline{\Psi})).  
 \end{array}
\] 
Since $\m^{[p^{e}]} \subseteq I_{e}$ and $\coker(\overline{\Psi})$ is a quotient of $(R/I_{e})^{\oplus p^\gamma}$, we have that $\m^{[p^e]} \subseteq \Ann_R(\coker(\overline{\Psi}))$.  But $\coker(\overline{\Psi})$ is also a quotient of $\coker(\Psi) = N$ and thus $N / \m^{[p^e]}N$, so that $\ell_R(\coker(\overline{\Psi})) \leq \ell_R(N/\m^{[p^e]}N) \leq C(N) p^{e(d-1)}$ by Lemma~\ref{Well known bound}.  Dividing through by $[k^{1/p}:k]p^{(e+1)d} = p^{\gamma + ed}$ yields
\[
 \frac{1}{p^{ed}}\ell_R(R/I_e) \leq \frac{1}{p^{(e+1)d}} \ell_R(R/I_{e+1}) + \frac{C(N,\m) / p^\gamma}{p^e}
\]
and the result now follows from Lemma~\ref{Sequence Lemma} with $C(\psi) = 2C(N,\m) / p^\gamma$ independent of the sequence of ideals $\{I_e\}_{e \in \N}$ satisfying the condition in (ii.).  
\end{proof}

\begin{Corollary}\label{Limits exist combined} Let $(R, \m, k)$ be an $F$-finite local domain of dimension $d$, and $\{ I_e \}_{e \in \N}$ a sequence of ideals such that $\m^{[p^e]} \subseteq I_e$ for all $e \in \N$.  Suppose there exists $0 \neq c \in R$ so that $c I_e^{[p]} \subseteq I_{e + 1}$ for all $e \in \N$, and a non-zero $R$-linear map $\psi \: R^{1/p} \to R$ so that $\psi(I_{e+1}^{1/p}) \subseteq I_e$ for all $e \in \N$. Then $\eta = \lim_{e \to \infty} \frac{1}{p^{ed}} \ell_R(R/I_e)$ exists, and there is a positive constant $C \in \R$ such that $|\frac{1}{p^{ed}} \ell_R(R/I_e) - \eta |\leq \frac{C}{p^e}$ for all $ e \in \N$ and all such sequence of ideals $\{I_e\}_{e \in \N}$.  In particular, $\ell_R(R/I_e)=\eta p^{ed}+O(p^{e(d-1)})$.
\end{Corollary}

Corollary~\ref{Limits exist combined} gives yet another perspective on the existence proofs for Hilbert-Kunz multiplicity and $F$-signature in terms of the properties of certain sequences of ideals.  If $(R,\m,k)$ is a local domain and we set $I_e^{\mathrm{HK}} = \m^{[p^e]}$ and $I_e^{\mathrm{F-sig}} = ( r \in R \mid \phi(r^{1/p^e}) \in \m \mbox{ for all } \phi \in \Hom_R(R^{1/p^e},R) )$, it is easy to see that $\mu_R(R^{1/p^e}) = \ell_R((R/I_e^{\mathrm{HK}})^{1/p^e})$ and $\frk_R(R^{1/p^e}) = \ell_R((R/I^{\mathrm{F-sig}}_e)^{1/p^e})$.  It is shown in \cite{Tucker2012} that both sequences satisfy the conditions of Theorem~\ref{Limits exist combined}; in fact, any choice of $0 \neq c \in R$ and $0 \neq \psi \in \Hom_R(R^{1/p},R)$ will suffice.

As in \cite{Tucker2012}, in settings such as Theorem~\ref{MainExistenceTheoremIdealSequences}, we have opted to consider sequences of ideals $\{I_e\}_{e \in \N}$ with $\m^{[p^e]} \subseteq I_e$.  In light of Theorem~\ref{polstrabound}, uniformity of constants over $\Spec(R)$ would seem more transparent for such sequences.  Nonetheless, for any $\m$-primary ideal $J$, one could consider sequences with $J^{[p^e]} \subseteq I_e$.  One has $\m^{[p^{e_0}]} \subseteq J$ for some $e_0 \in \N$, so that $\m^{[p^{e + e_0}]} \subseteq J^{[p^e]} \subseteq I_e$ for such sequences.  In particular, our techniques  carry over to this setting after allowing for reindexing, and allows one to recover Monsky's original limit existence result.

\begin{Corollary}
\cite{Monsky83}
If $(R,\m,k)$ is a local domain of dimension $d$ and $J \subseteq R$ is an $\m$-primary ideal, then 
\[
e_{HK}(R, J) = \lim_{e \to \infty} \frac{1}{p^{ed}} \ell_R(R / J^{[p^e]})
\]
exists and there is a positive constant $C \in \R$ such that
\[
|e_{HK}(R, J)  -  \frac{1}{p^{ed}} \ell_R(R / J^{[p^e]})| < \frac{C}{p^e}
\]
for all $e$.
\end{Corollary}

Another useful modification of the setup of Theorem~\ref{MainExistenceTheoremIdealSequences} proceeds in the following manner. Fixing $e_0 \in \N$, one can modify the condition on the sequences of ideals $\{I_e\}_{e \in \N}$ of Theorem~\ref{MainExistenceTheoremIdealSequences} (i.) to require that $c I_e^{[p^{e_0}]} \subseteq I_{e+e_0}$ for all $e \in \N$.  Similarly, in Theorem~\ref{MainExistenceTheoremIdealSequences} (ii.), one can ask that there exists $0 \neq \psi \in \Hom_R(R^{1/p^{e_0}},R)$ such that $\psi( (I_{e+e_0})^{1/p^{e_0}}) \subseteq I_e$ for all $e \in \N$.  In both cases,  after reindexing, the same methods apply to yield analogous results.  For example, this last generalization is particularly relevant when working with an arbitrary Cartier subalgebra $\sD$ on an local $F$-finite domain  $(R, \m, k)$. If $\Gamma_\sD = \{ e \in \N \, | \, \mbox{ such that } \sD_e \neq 0 \}$ is the semi-group of $\sD$ and $e \in \Gamma_\sD$, one  defines the $e$-th $F$-splitting number $a_e^\sD$ of $R$ along $\sD$ to be the maximal number of copies of $R$ with projection maps in $\sD_e$ appearing in an $R$-module direct sum decomposition of $R^{1/p^e}$. Setting $I_e^\sD = ( r \in R \mid \phi(r^{1/p^e}) \in \m \mbox{ for all } \phi \in \sD_e)$,  it is again easy to check that $a_e^\sD = \ell_R((R/I_e^\sD)^{1/p^e})$, and the method of Theorem~\ref{MainExistenceTheoremIdealSequences} (ii.) yields the following result.

\begin{Theorem}
\label{Dsignatureexists}
Let $(R, \m, k)$ be an $F$-finite local domain of dimension $d$ and $\sD$ a Cartier subalgebra on $R$.  Then the $F$-signature  
$s(R,\sD) = \lim_{ \substack{e \to \infty \\ e \in \Gamma_\sD}} \frac{1}{[k^{1/p^e}:k] \cdot p^{ed}} a_e^\sD$ of $R$ along $\sD$ exists, and there is a positive constant $C \in \R$ so that $$s(R, \sD) - \frac{1}{[k^{1/p^e}:k] \cdot p^{ed}}a_e^\sD \leq \frac{C}{p^e}$$ for all $e \in \Gamma_\sD$.
\end{Theorem}

\begin{proof} 
Let $p^\gamma = \rk_R(R^{1/p}) = [k^{1/p}:k] \cdot p^d$ and take a set\footnote{Every sub-semigroup of $\N$ is finitely generated \cite[Proposition 4.1]{Grillet}.} of generators $e_1, \ldots, e_s$ for the semigroup $\Gamma_\sD$.  Fixing $0 \neq \psi_i \in \sD_{e_i}$ for each $i = 1, \ldots, s$, we can find a short exact sequence of $R$-modules
\begin{equation}
\label{Dshortexactsequence}
\xymatrix{
0 \ar[r] & R^{1/p^{e_i}} \ar[r]^{\Psi_i} & R^{\oplus p^{\gamma e_i}} \ar[r] & M_i \ar[r] & 0
}
\end{equation}
 where $\dim(M_i) < d$ and every component function of $\Psi_i$ is a pre-multiple of $\psi_i$ as in Lemma~\ref{Map Lemma} (ii.).  Since the Cartier linear maps in $\sD$ are closed under composition, it follows readily that $\psi_i((I_{e+e_i}^\sD)^{1/p^{e+e_i}}) \subseteq I_e^\sD$ for any $e \in \Gamma_\sD$.  In particular, $\Psi_i((I_{e+e_i}^\sD)^{1/p^{e+e_i}}) \subseteq (I_e^\sD)^{\oplus p^{\gamma e_i}}$ and proceeding as in Theorem~\ref{MainExistenceTheoremIdealSequences} (ii.), we see that
\[
\frac{1}{p^{(e + e_i)\gamma}} a_{e+e_i}^\sD \leq \frac{1}{p^{e \gamma}} a_e^\sD + \frac{C'}{p^e}
\]
for any $i = 1, \ldots, s$ and $e \in \Gamma_\sD$ where $C' = \max \{ C(M_1, \m)/ p^{e_1 \gamma}, \ldots, C(M_s, \m) / p^{e_s \gamma} \}$. It follows as in Lemma~\ref{Sequence Lemma} (ii.) that $s(R,\sD) = \lim_{ \substack{e \to \infty \\ e \in \Gamma_\sD}} \frac{1}{p^{e\gamma}} a_e^\sD$ exists and $s(R, \sD) - \frac{1}{p^{e\gamma}}a_e^\sD \leq \frac{2C'}{p^e}$ for all $e \in \Gamma_\sD$ as desired.
\end{proof}

\begin{Remark}[\textit{cf.} \cite{PolstraSemicont} condition (1)]
Consider  a Cartier subalgebra $\sD$ on an $F$-finite local domain $(R, \m, k)$ with the following property: if $\phi \in \sD_{e+1}$ for some $e \in \N$, then $\phi|_{R^{1/p^e}} \in \sD_e$.  It is easy to check $(I_e^\sD)^{[p]} \subseteq I_{e+1}^\sD$ for all $e\in \N = \Gamma_\sD$, and Theorem~\ref{MainExistenceTheoremIdealSequences} (i.) gives that there is a positive constant $C \in \R$ so that 
\begin{equation}
\label{strongdbounds}
\left| s(R, \sD) - \frac{1}{[k^{1/p^e}:k] \cdot p^{ed}}a_e^\sD \right| \leq \frac{C}{p^e}
\end{equation} for all $e \in \N$.  Arbitrary Cartier subalgebras will not satisfy such properties, and we know of no reason to expect \eqref{strongdbounds} to hold in general.  However, we will see below in Theorem~\ref{ateefsigexists} and Theorem~\ref{deltafsigexists} that the Cartier subalgebras constructed using ideals and divisors do satisfy \eqref{strongdbounds} by showing they enjoy a perturbation of the above property.
\end{Remark}

The proof of Theorem~\ref{MainExistenceTheoremIdealSequences} made use of positive constants arising from Lemma~\ref{Well known bound}.  However, in case $R$ is not local and the $p$-linear map in (i.) or $p^{-1}$-linear map in (ii.) extend to $R$, one can make these constants spread uniformly over $\Spec(R)$ using Theorem~\ref{polstrabound} instead.  In the case of Theorem~\ref{Dsignatureexists}, this immediately yields the following result.

\begin{Theorem}
\label{Dsignaturesemicont}
Let $R$ be a locally equidimensional $F$-finite domain and $\sD$ a Cartier subalgebra on $R$.  There is a positive constant $C \in \R$ so that $$s(R_\p, \sD_\p) - \frac{1}{[k(\p)^{1/p^e}:k(\p)] \cdot p^{e\height(\p)}}a_e^{\sD_\p} \leq \frac{C}{p^e}$$ for all $e \in \Gamma_\sD$ and all $\p \in \Spec(R)$. Moreover, the function $\Spec(R) \to \R$ given by $\p \mapsto s(R_\p, \sD_\p)$ is lower semicontinuous. 
\end{Theorem}

\noindent
One can easily relax the requirements on $R$ above -- any $F$-finite ring will suffice -- by using the fact that the $F$-signature function is identically zero off of the strongly $F$-regular locus of $R$ (and hence $s(R_\p, \sD_\p) = 0$  for any $\p \in \Spec(R)$ so that $R_\p$ is a Cohen-Macaulay normal domain).

\begin{proof}
Following along in the proof of Theorem~\ref{Dsignatureexists}, observe that sequences in \ref{Dshortexactsequence} can be taken to be global and then localized to each $\p \in \Spec(R)$.  By Theorem~\ref{polstrabound}, we may then take the positive constant $C = \max \{ 2C(M_1)/ p^{e_1 \gamma}, \ldots, 2C(M_s) / p^{e_s \gamma} \}$ independent of $\p \in \Spec(R)$.  For the remainder, first note the argument of Lemma~\ref{easysemicont} readily adapts to show the function $\Spec(R) \to \R$ given by $\p \mapsto \a_e^{\sD_\p}$ is lower semicontinuous. The lower semicontinuity of the $F$-signatures thus follows from the argument given in Theorem~\ref{Fsigsemicont}.  
\end{proof}

\begin{Remark}
Analogous to the argument above, for an $\R$-valued function on $\Spec(R)$ governed locally by sequences of ideals as in Theorem~\ref{MainExistenceTheoremIdealSequences} (i.), upper semicontinuity passes from the individual terms in the sequences to the limit as in Theorem~\ref{HKsemicont}.  However, we are unaware of any such functions that are not directly related Hilbert-Kunz multiplicity itself. 
\end{Remark}

\begin{Theorem}
\label{ateefsigexists}
Let $R$ be an $F$-finite domain, $\a \subseteq R$ a non-zero ideal, $t \in \R_{\geq 0}$, and $\p \in \Spec(R)$. Then the $F$-signature $$s(R_\p, \a_\p^t) = s(R_\p, \sC^{\a_p^t}) = \lim_{e \to \infty} \frac{1}{[k(\p)^{1/p^e}:k(\p)] \cdot p^{e\height(\p)}} a_e^{\a_\p^t}$$ of $R_\p$ along $\a_\p^t$ exists and determines a lower semicontinuous $\R$-valued function on $\Spec(R)$. Moreover, there is a positive constant $C \in \R$ so that
\begin{equation}
\label{ateeohpdeeminus1bound}
\left| s(R_\p, \a_\p^t) - \frac{a_e^{\a_\p^t}}{[k(\p)^{1/p^e}:k(\p)] \cdot p^{e\height(\p)}}  \right|
 \leq \frac{C}{p^e}
\end{equation}
for all $e \in \N$, all $\p \in \Spec(R)$, and all $t \in \R_{\geq 0}$.  In particular, $\frac{a_e^{\a_\p^t}}{[k(\p)^{1/p^e}:k(\p)]} =  s(R_\p, \a_\p^t) p^{e\height(\p)} + O(p^{e(\height(\p)-1)})$.
\end{Theorem}

\begin{proof} The existence and semicontinuity statements follow immediately from Theorems~\ref{Dsignatureexists} and \ref{Dsignaturesemicont} above.
Let $f_1, \ldots, f_s$ be a set of generators for $\a$.  If $t \geq s$, then $a_e^{\a_\p^t} = 0$ for all $e \in \N$ and any $\p \in \Spec(R)$.  Thus, we are free to assume $t < s$ going forward.  

For a fixed $t \in \R$, one inequality in \eqref{ateeohpdeeminus1bound} also follows from Theorem~\ref{Dsignatureexists} above; however, it remains to show that the positive constant $C$ can be chosen independent of $t \in \R$.  To that end, fix a choice of $0 \neq x \in \a^{s(p-1)}$ and $\phi \in \Hom_R(R^{1/p},R)$.  Then $\psi(\blank) = \phi(x^{1/p} \cdot \blank) \in \sC_1^{\a^t}$ for any $t < s$.  The desired independence follows from the observation that a short exact sequence as in Lemma~\ref{Map Lemma} (ii.) can be used for \eqref{Dshortexactsequence} in the proof of Theorem~\ref{Dsignatureexists} for any $t < s$.

To show the reverse inequality, choose an element $0 \neq c \in R$ such that $c \a^{\lceil t (p^{e+1} - 1) \rceil} \subseteq (\a^{\lceil t(p^e - 1)\rceil})^{[p]}$ for all $e \in \N$ and any $t < s$; one can check that $c = f_1^{p} \cdots f_s^{p}$ will suffice. We will show that $c (I_e^{\a_\p^t})^{[p]} \subseteq I_{e+1}^{\a_\p^t}$ for all $t < s$ and $\p \in \Spec(R)$, after which the result follows by using a short exact sequence as in Lemma~\ref{Map Lemma} (i.)  for \eqref{firstshortexactinmainexistence} in the proof of Theorem~\ref{MainExistenceTheoremIdealSequences} (i.) for any $t < s$ and $\p \in \Spec(R)$ with $C(M_\p,\p R_\p) = C(M)$ from Theorem~\ref{polstrabound}. Supposing that $\phi \in \Hom_{R_\p}(R_\p^{1/p^{e+1}},R_\p)$ and $y \in \a_\p^{\lceil t(p^e - 1)\rceil}$, we must show $\phi((yc (I_e^{\a_\p^t})^{[p]} )^{1/p^{e+1}}) \subseteq \p R_\p$.  Write $cy = \sum_{i=1}^n g_i^p r_i$ where each $r_i \in R_\p$ and $g_i \in \a_\p^{\lceil t(p^e - 1) \rceil}$.  Let $\phi_i \in \Hom_{R_\p}(R_\p^{1/p^e},R_\p)$ be the map defined by $\phi_i(z^{1/p^e}) = \phi(r_i^{1/p^{e+1}}z^{1/p^e})$ for all $z \in R_\p$.  If $x \in I_e^{\a_\p^t}$, then $\phi((ycx^p)^{1/p^{e+1}}) = \sum_{i=1}^n \phi(r_i^{1/p^{e+1}} (g_i x)^{1/p^e}) = \sum_{i=1}^n \phi_i((g_i x)^{1/p^e}) \in \p R_\p$ as $\phi_i( g_i^{1/p^e}\cdot \blank) \in \sC^{\a_\p^t}_e$ for all $i = 1, \ldots, n$.
\end{proof}

\begin{Theorem}
\label{deltafsigexists}
Let $R$ be an $F$-finite normal domain, $\Delta$ an effective $\Q$-divisor on $\Spec(R)$, and $\p \in \Spec(R)$. Then the $F$-signature $$s(R_\p, \Delta) = s(R_\p, \sC^{(R_\p,\Delta)}) = \lim_{e \to \infty} \frac{1}{[k(\p)^{1/p^e}:k(\p)] \cdot p^{e\height(\p)}} a_e^{(R_\p,\Delta)}$$ of $R_\p$ along $\Delta$ exists and determines a lower semicontinuous $R$-valued function on $\Spec(R)$. Moreover, if $\overline{\Delta}$ is a fixed effective integral divisor on $\Spec(R)$, there exists a positive constant $C \in \R$ so that
\begin{equation}
\label{deltaohpdeeminus1bound}
\left| s(R_\p, \Delta) - \frac{a_e^{(R_\p,\Delta)}}{[k(\p)^{1/p^e}:k(\p)] \cdot p^{e\height(\p)}}  \right|
 \leq \frac{C}{p^e}
\end{equation}
for all $e \in \N$, all $\p \in \Spec(R)$, and all effective $\Q$-divisors $\Delta$ with $\Delta \leq\overline{\Delta}$.  In particular, we have that $\frac{a_e^{(R_\p,\Delta)}}{[k(\p)^{1/p^e}:k(\p)]} =  s(R_\p, \Delta) p^{e\height(\p)} + O(p^{e(d-1)})$.
\end{Theorem}

\begin{proof}
The existence and semicontinuity statements follow immediately from Theorems~\ref{Dsignatureexists} and \ref{Dsignaturesemicont} above.
For a fixed $\Delta$, one inequality in \eqref{deltaohpdeeminus1bound} also follows from Theorem~\ref{Dsignatureexists} above; however, it remains to show that the positive constant $C$ can be chosen independent of $\Delta \leq \overline{\Delta}$.  To that end, fix a choice of 
$0 \neq \psi \in \sC_1^{(R,\overline \Delta)}$.  Since $\sC_1^{(R,\overline \Delta)} \subseteq \sC_1^{(R,\Delta)}$, the desired independence follows from the observation that a short exact sequence as in Lemma~\ref{Map Lemma} (ii.) can be used for \eqref{Dshortexactsequence} in the proof of Theorem~\ref{Dsignatureexists} for any $\Delta \leq \overline{\Delta}$.

For the reverse inequality, choose $0 \neq c \in R$ so that $\mathrm{div}_R(c) \geq p\overline{\Delta}$.  We will show that $c (I_e^{(R_\p,\Delta)})^{[p]} \subseteq I_{e+1}^{(R_\p,\Delta)}$ for all $\Delta \leq \overline \Delta$, after which the result follows by using a short exact sequence as in Lemma~\ref{Map Lemma} (i.)  for \eqref{firstshortexactinmainexistence} in the proof of Theorem~\ref{MainExistenceTheoremIdealSequences} (i.) for any $\Delta \leq \overline \Delta$ and $\p \in \Spec(R)$ with $C(M_\p,\p R_\p) = C(M)$ from Theorem~\ref{polstrabound}.  
Supposing $\p \in \Spec(R)$ and $\phi \in \sC_{e+1}^{(R_\p,\Delta)}$, 
we must show $\phi((c (I_e^{(R_\p,\Delta)})^{[p]} )^{1/p^{e+1}}) \subseteq \p R_\p$.  Let $\psi \in \Hom_{R_\p}(R_\p^{1/p^{e+1}},R_\p)$ be the map given by $\psi(\blank) = \phi(c^{1/p^{e+1}} \cdot \blank)$, and $\psi_e = \psi|_{R_\p^{1/p^e}} \in \Hom_{R_\p}(R_\p^{1/p^e},R_\p)$ be the restriction to $R_\p^{1/p^e}$.  It suffices to show $\Delta_{\psi_e} \geq \Delta$, so that for any $x \in I_e^{(R_\p,\Delta)}$ we have $\phi((cx^p)^{1/p^{e+1}}) = \psi_e(x^{1/p^e}) \in \m$.

To that end, consider the following inclusions for each $e \in \N$
\[
\begin{array}{ccc}
\; \; R^{1/p^e} & \rotatebox[origin=c]{0}{$\subseteq$} & \; \; R^{1/p^{e+1}} \\
 \rotatebox[origin=c]{-90}{$\subseteq$}& &  \rotatebox[origin=c]{-90}{$\subseteq$} \\
(R(\lceil (p^e -1)\Delta)\rceil)^{1/p^e} & \rotatebox[origin=c]{0}{$\subseteq$} & \; \; \; \; \; (R(p \lceil(p^e-1)\Delta \rceil))^{1/p^{e+1}} \\
& & \rotatebox[origin=c]{-90}{$\subseteq$} \\ 
& & \; \; \; \; (R( \lceil(p^{e+1}-1)\Delta \rceil  + \mathrm{div}_R(c)))^{1/p^{e+1}}
\end{array}
\]
which hold because $\lceil (p^e - 1) \Delta \rceil  \leq  (p^e - 1) \Delta + \overline\Delta$ implies
\[
\begin{array}{rcl}
p\lceil (p^e - 1) \Delta \rceil & \leq & (p^{e+1} - p) \Delta + p\overline\Delta\\
& \leq & (p^{e+1} - 1) \Delta + p\overline\Delta\\
& \leq & \lceil (p^{e+1} - 1) \Delta \rceil + \mathrm{div}_R(c) .
\end{array}
\]
We know that $\psi$ is the restriction to $R_\p^{1/p^{e+1}}$ of a map in $$\Hom_{R_\p}((R_\p( \lceil(p^{e+1}-1)\Delta \rceil  + \mathrm{div}_R(c)))^{1/p^{e+1}},R_\p),$$ and localizing the above inclusions at $\p$ we see that $\psi_e = \psi|_{R_\p^{1/p^e}}$ can be extended to a map in $\Hom_{R_\p}((R_\p(\lceil (p^e -1)\Delta)\rceil)^{1/p^e} ,R_\p)$.  
It follows immediately that $\Delta_{\psi_e} \geq \Delta$.
\end{proof}

As the above examples demonstrate, Theorem~\ref{MainExistenceTheoremIdealSequences} can be used to show the existence of limits in a large number of important settings. Moreover, the following well-known lemma allows similar techniques to be used more broadly still.

\begin{Lemma}
\label{colonsamelimit}
Let $(R, \m, k)$ be an $F$-finite local domain of dimension $d$ and $0 \neq c \in R$.  Suppose $\{I_e\}_{e \in \N}$ and $\{J_e\}_{e \in \N}$ are two sequences of ideals in $R$ so that 
\[
\m^{[p^e]} \subseteq I_e\subseteq J_e\subseteq (I_e:c)
\]
for all $e \in \N$. Then there exists a positive constant $C \in \R$ so that 
\[
\left| \frac{1}{p^{ed}} \ell_R(R/I_e) - \frac{1}{p^{ed}} \ell_R(R/J_e)\right| = \frac{1}{p^{ed}} \ell_R(J_e / I_e) \leq \frac{C}{p^e}
\]
for all $e \in \N$ and so 
$\lim_{e \to \infty} \left( \frac{1}{p^{ed}} \ell_R(R/I_e) - \frac{1}{p^{ed}} \ell_R(R/J_e) \right) = 0$. Hence $\lim_{e \to \infty}\frac{1}{p^{ed}}\ell_R(R/J_e)$ exists if and only if $\lim_{e \to \infty}\frac{1}{p^{ed}}\ell_R(R/I_e)$ exists, in which case they are equal.
\end{Lemma}

\begin{proof}
Note that as $(I_e:c)/I_e$ is the kernel and $R / (I_e, c)$ the cokernel of multilication by $c$ on $R/I_e$, their lengths are equal.  Thus
\[
\begin{array}{rcl}
\frac{1}{p^{ed}} \ell_R(J_e / I_e)  &\leq & \frac{1}{p^{ed}} \ell_R((I_e:x) / I_e)  \\
& = & \frac{1}{p^{ed}} \ell_R(R /( I_e,c)) \\
& \leq & \frac{1}{p^{ed}} \ell_R(R /( \m^{[p^e]},c)) 
 \leq \frac{C}{p^e}
\end{array}
\]
by applying Lemma~\ref{Well known bound} with $M = R/(c)$.
\end{proof}

\noindent
Note that, once again, the constant $C$ in the above lemma depends only on the choice of $0 \neq c \in R$ and not on any particular sequence of ideals; moreover, in case $R$ is not local, the constant $C$ can be chosen uniformly over $\Spec(R)$ using Theorem~\ref{polstrabound}.

\section{Positivity}\label{Section Positivity}  

\subsection{Positivity of the F-signature}
Recall that an $F$-finite ring $R$ is said to be strongly $F$-regular if and only if, for all $0 \neq x \in R$ there exists some $e \in \N$ and $\phi \in \Hom_R(R^{1/p^e},R)$ with $\phi(x^{1/p^e}) = 1$.  Strongly $F$-regular rings are always Cohen-Macaulay normal domains, and remain strongly $F$-regular after completion.  In particular, the completion of a strongly $F$-regular local domain remains a domain.

\begin{Theorem}\label{Aberbach Leuschke}
\cite[Main Result]{AbLeu}
 Suppose $(R, \m, k)$ is a local $F$-finite domain.  Then $R$ is strongly $F$-regular if and only if $s(R) > 0$.
\end{Theorem}

We will give two new proofs of Theorem~\ref{Aberbach Leuschke}.  The first proof is notable in that it gives a readily computable lower bound for the $F$-signature.

\begin{proof}[First proof of Theorem~\ref{Aberbach Leuschke}]
We will assume for simplicity in the exposition that $k = k^p$ is perfect -- the proof is easily adapted to arbitrary $k$.  Suppose first that $R$ is not strongly $F$-regular, so that there exists some $0 \neq x \in R$ with $\phi(x^{1/p^e}) \in \m$ for all $e \in \N$ and all $\phi \in \Hom_R(R^{1/p^e},R)$.  Take a surjection $\Phi \: R^{1/p^e} \to R^{\oplus \frk_R(R^{1/p^e})}$ of $R$-modules.  It follows that $\Phi((xR)^{1/p^e}) \subseteq \m^{\oplus \frk_R(R^{1/p^e})}$, so that $\Phi$ induces a surjection of $R$-modules $(R / xR)^{1/p^e} \to k^{\oplus \frk_R (R^{1/p^e})}$ and hence $\frk_R(R^{1/p^e}) \leq \mu_R((R / xR)^{1/p^e}) \leq C(R/xR, \m) p^{e(d-1)}$ using \eqref{mingenstoHK} and Lemma~\ref{Well known bound} for some positive constant $C(R/xR, \m) \in \R$. Dividing through by $p^{ed}$ and taking $\lim_{e \to \infty}$ then gives $s(R) = 0$.  Thus, $s(R) > 0$ implies that $R$ is strongly $F$-regular.

{
\newcommand{\qz}{p^{e_0}}
\newcommand{\qqz}{p^{e+ e_0}}
\renewcommand{\q}{p^{e}}
\newcommand{\x}{\underline{x}}
\newcommand{\sI}{\mathcal{I}}
Conversely, suppose that $R$ is strongly $F$-regular.  The $F$-signature remains unchanged upon completion, so we may assume $R$ is complete and choose a coefficient field $k$ and system of parameters $x_1, \ldots, x_d$  so that $R$ is module finite and generically separable over the regular subring $A = k[[ x_1, \ldots, x_d]]$.  Take $0 \neq c \in A$ so that $c R^{1/p^e} \subseteq R[A^{1/p^e}] = R \otimes_A A^{1/p^e}$ as in Lemma~\ref{ranklemma} for all $e \in \N$.  Since $R$ is strongly $F$-regular, we can find $e_0 \in \N$ and $\phi \in \Hom_R(R^{1/p^{e_0}},R)$ with $\phi(c^{1/p^{e_0}}) = 1$.  We will show $s(R) \geq \frac{1}{p^{e_0d}} > 0$.

For any $\alpha = (\alpha_1, \ldots, \alpha_d) \in \mathbb{Z}^d$ and $e \in \N$, we write $\underline{x}^{\alpha / \q}$ for $x_1^{\alpha_1 / \q} \cdots x_d^{\alpha_d / \q}$. 
Let $\sI_e = \{ \alpha = (\alpha_1, \ldots, \alpha_d) \in \mathbb{Z}^d \; | \; 0 \leq \alpha_i < \q \mbox{ for } i = 1, \ldots, d \}$.  
The monomials $\x^{\alpha / \q}$ for $\alpha \in \sI_{e}$ are a free basis for $A^{1/\q}$ over $A$.  As such, for each $\alpha \in \sI_e$, we can find an $A$-linear map $\pi_\alpha \: A^{1/\q} \to A$ with $\pi_\alpha(\x^{\alpha/ \q}) = 1$ and $\pi_\alpha(\x^{\beta/\q}) = 0$ for all $\alpha \neq \beta \in \sI_e$.  Applying $R \otimes_A \blank$ gives an $R$-linear map $\tilde\pi_\alpha \: R[A^{1/\q}] \to R$ with $\tilde\pi_\alpha(\x^{\alpha/ \q}) = 1$ and $\tilde\pi_\alpha(\x^{\beta/\q}) = 0$ for all $\alpha \neq \beta \in \sI_e$. 
Multiplication by $c$ gives an $R$-linear map $R^{1/\q} \to R[A^{1/\q}]$, and composing with $\tilde\pi_\alpha$ gives an $R$-linear map $\tilde\pi_\alpha( c \cdot \blank) \: R^{1/\q} \to R$ so that 
$\tilde\pi_\alpha(c\, \x^{\alpha/ \q}) = c \,\tilde\pi_\alpha(\x^{\alpha/ \q}) = c $ 
and $\tilde\pi_\alpha(c\, \x^{\beta/\q}) = c \, \tilde\pi_\alpha( \x^{\beta/\q}) = 0$ for all $\alpha \neq \beta \in \sI_e$ (since $c \in A$).  Setting $\psi_\alpha( \blank) = \phi \circ (\tilde\pi_\alpha( c \cdot \blank))^{1/ p^{e_0}} \: R^{1/\qqz} \to R$, we have
$\psi_\alpha \in \Hom_R(R^{1/\qqz},R)$ with $\psi_\alpha(\x^{\alpha/ \qqz}) = 1$ and $\psi_\alpha(\x^{\beta/\qqz}) = 0$ for all $\alpha \neq \beta \in \sI_e$.  It follows that $\Psi = \oplus_{\alpha \in \sI_e} \psi_\alpha \: R^{1/\qqz} \to R^{\oplus p^{ed}}$ is an $R$-module surjection and hence $\frk_R(R^{1/\qqz}) \geq p^{ed}$.  Dividing through by $p^{(e+e_0)d}$ and taking $\lim_{e \to \infty}$ yields $s(R) \geq \frac{1}{p^{e_0d}} > 0$ as desired.
}
\end{proof}

For the second proof, we will need the following lemma, which should be compared with \cite[Theorem 3.3]{HHEltsSmallOrder}.  In the next subsection, we will show how to adapt this proof to arbitrary sequences of ideals as in Theorem~\ref{MainExistenceTheoremIdealSequences} (ii.).

\begin{Lemma}\label{HHLemma} Let $(R,\m,k)$ be a complete local F-finite Cohen-Macaulay domain of dimension $d$. There exist  $N \in \N$ with the following property: for any $e \in \N$ and all $x \in R \setminus \m^{[p^e]}$, there exists a map $\phi \in \Hom_R(R^{1/p^e},R)$ with $\phi(x^{1/p^e}) \not\in \m^N$.
\end{Lemma}

\newcommand{\Tr}{\mathrm{Tr}}
\begin{proof}
Choosing a coefficient field $k$ and system of parameters $x_1, \ldots, x_d$, we have that $R$ is a finitely generated free module over $A = k[[ x_1, \ldots, x_d]]$ (see \cite[Proposition 2.2.11]{BrunsHerzog}); let $\m_A$ denote the maximal ideal of $A$.  Fix a non-zero map $\tau \in \Hom_A (R,A)$.  Since $\Hom_A(R,A)$ is a rank one torsion free $R$-module, we can find find $0 \neq y \in \Ann_R \left( \Hom_A(R,A)/(R  \tau)\right)$; replacing by a nonzero multiple, we may further assume $0 \neq y \in A$.  If $N',N \in \N$ are sufficiently large so that $y \not\in \m_{A}^{N'}$ and $\m^N \subseteq \m_A^{N'}R$, we will show $N$ satisfies the desired property.

Fix $e \in \N$ and $x \in R \setminus \m^{[p^e]}$, so also $x \not\in \m_A^{p^e}R$. Since $R^{1/p^e}$ is a free $A$-module and $x^{1/p^e} \not\in \m_A R^{1/p^e}$,  $x^{1/p^e}$ can be taken as part of a free basis and there exists an $A$-linear map $\chi \: R^{1/p^e} \to A$ with $\chi(x^{1/p^e}) = 1$. If $\Tr \: \Hom_A(R,A) \to A$ is the $A$-linear map given by evaluation at $1$, there is a commutative diagram
\begin{equation}
\label{bigmapdiagram}
\xymatrix{
R^{1/p^e} \ar[drrr]_-{\chi} \ar[rr]^-{\tilde{\chi}}  && \Hom_A(R^{1/p^e},A) \ar[r] & \Hom_A(R,A) \ar[rr]^{\cdot y} \ar[d]^{\Tr}  & & R\tau \ar[rr]^{\simeq} \ar[d]^{\Tr} & & R  \ar[d]^{\tau}\\
	  &    &  & A 	\ar[rr]^-{\cdot y}		& & A \ar[rr]^{=}		& & A
}
\end{equation}
where $\tilde{\chi}$ is the $R^{1/p^e}$-linear map sending $1 \mapsto \chi$ and $\Hom_A(R^{1/p^e},A) \to \Hom_A(R,A)$ is the $R$-linear map given by restriction to $R \subseteq R^{1/p^e}$.  Let $\phi \in \Hom_R(R^{1/p^e},R)$ be the composition along the top row of this diagram, so that $y \chi =  \tau \circ \phi$.  We have $\tau(\phi(x^{1/p^e})) = y \chi(x^{1/p^e}) = y \not\in \m_A^{N'}$, and thus $\phi(x^{1/p^e}) \not\in \m_A^{N'}R$ as $\tau$ is $A$-linear.  In particular, since $\m^N \subseteq \m_A^{N'}R$, we have $\phi(x^{1/p^e}) \not\in \m^N$ as desired.
\end{proof}

\begin{proof}[Second proof of Theorem~\ref{Aberbach Leuschke}]
Suppose that $R$ is strongly $F$-regular, and thus Cohen-Macaulay.  As before, without loss of generality, we may assume $R$ is complete. If $I_e^{\mathrm{F-sig}} = ( r \in R \mid \phi(r^{1/p^e}) \in \m \mbox{ for all } \phi \in \Hom_R(R^{1/p^e},R) )$ and $\psi \in \Hom_R(R^{1/p^{e'}},R)$, it is easy to check that $\psi((I_{e+e'}^{\mathrm{F-sig}})^{1/p^{e'}}) \subseteq I_e^{\mathrm{F-sig}}$ for all $e, e' \in \N$.  Since $R$ is $F$-split, we may conclude $I_{e+1}^{\mathrm{F-sig}} \subseteq I_{e}^{\mathrm{F-sig}}$ for all $e \in \N$; moreover, by the definition of strong $F$-regularity, we have $\bigcap_{e \in \N} I_e^{\mathrm{F-sig}} = 0$.  Thus, for $N$ as in Lemma~\ref{HHLemma}, Chevalley's Lemma \cite[Lemma 7]{ChevalleysLemma} gives $e_0 \in \N$ with $I_{e_0}^{\mathrm{F-sig}} \subseteq \m^N$.  It follows from Lemma~\ref{HHLemma} that $I_{e+e_0}^{\mathrm{F-sig}} \subseteq \m^{[p^e]}$ for all $e \in \N$.  Indeed for each $x \in R \setminus \m^{[p^e]}$ there is some $\phi \in \Hom_R(R^{1/p^e},R)$ with $\phi(x^{1/p^e}) \not\in \m^N$ and hence $\phi(x^{1/p^e}) \not\in I_{e_0}^{\mathrm{F-sig}}$; as $\phi((I_{e+e_0}^{\mathrm{F-sig}})^{1/p^{e}}) \subseteq I_{e_0}^{\mathrm{F-sig}}$ we must have $x \not\in I_{e+e_0}^{\mathrm{F-sig}}$.  Thus, we compute
\[
s(R) = \lim_{e \to \infty} \frac{1}{p^{(e+e_0)d}} \ell_{R}(R/I_{e+e_0}^{\mathrm{F-sig}}) \geq \lim_{e \to \infty } \frac{1}{p^{(e+e_0)d}} \ell_{R}(R/\m^{[p^e]}) = \frac{1}{p^{e_0 d}}e_{HK}(R) \geq \frac{1}{p^{e_0 d}} > 0
\]
as desired.
\end{proof}

\subsection{Positivity via a Cartier linear map}

In this section, we generalize the method of the second proof of Theorem~\ref{Aberbach Leuschke} to examine the positivity of the limits appearing in Theorem~\ref{MainExistenceTheoremIdealSequences} (ii.).    The essential technique is to exploit the following generalization of Lemma~\ref{HHLemma}, which should again be compared with \cite[Theorem 3.3]{HHEltsSmallOrder}.

\begin{Proposition}
\label{MainHHLemma}
Let $(R, \m, k)$ be a complete local $F$-finite domain of dimension $d$ and suppose that $0 \neq \psi \in \Hom_R(R^{1/p},R)$. There exist  $N \in \N$ and $0 \neq c \in R$ so that $\psi^e( (xR)^{1/p^e}) \not\subseteq \m^N$ for any $e \in \N$ and all $x \in R \setminus ( \m^{[p^e]} : c)$.
\end{Proposition}

\begin{proof}
Choose a coefficient field $k$ and system of parameters $x_1, \ldots, x_d$  so that $R$ is module finite and generically separable over the regular subring $A = k[[ x_1, \ldots, x_d]]$.  Take $0 \neq c \in A$ so that $c R^{1/p^e} \subseteq R[A^{1/p^e}] = R \otimes_A A^{1/p^e}$ as in Lemma~\ref{ranklemma} for all $e \in \N$. 

Fixing a non-zero map $\tau \in \Hom_A (R,A)$, we can find  $0 \neq y \in \Ann_R \left( \Hom_A(R,A)/(R  \tau)\right)$ as $\Hom_A(R,A)$ is a rank one torsion free $R$-module.  Let $0 \neq z \in R$ be as in Lemma~\ref{makeitaniterateofafixedmap}; replacing by nonzero multiples, we may further assume both $0 \neq y \in A$ and $0 \neq z \in A$. If $N',N \in \N$ are sufficiently large that $yz \not\in \m_{A}^{N'}$ and $\m^N \subseteq \m_A^{N'}R$, we will show the pair of $c$ and $N$ satisfy the desired property.

Suppose $e \in \N$ and $x \in R \setminus (\m^{[p^e]}: c)$. Since $xc \not\in \m_A^{[p^e]}R \subseteq \m^{[p^e]}$, we have $(xc)^{1/p^e} \not\in \m_A R^{1/p^e}$ and so also $(xc)^{1/p^e} \not\in \m_A R[A^{1/p^e}]$.  As $R[A^{1/p^e}] = R \otimes_A A^{1/p^e}$ is a free $A$-module, there is a $\chi' \in \Hom_A(R[A^{1/p^e}],A)$ with $\chi'((xc)^{1/p^e}) = 1$.  Let $\chi \in \Hom_A(R^{1/p^e},A)$ be the composition
\[
\xymatrix{
R^{1/p^e} \ar[r]^-{\cdot c}& R[A^{1/p^e}] \ar[r]^-{\chi'}& A
}
\]
which satisfies $\chi(x^{1/p^e}) = 1$.  If $\Tr \: \Hom_A(R,A) \to A$ is the $A$-linear map given by evaluation at $1$, the same diagram
\eqref{bigmapdiagram} from Lemma~\ref{HHLemma} commutes where $\tilde{\chi}$ is the $R^{1/p^e}$-linear map sending $1 \mapsto \chi$ and $\Hom_A(R^{1/p^e},A) \to \Hom_A(R,A)$ is the $R$-linear map given by restriction to $R \subseteq R^{1/p^e}$.  Let $\phi \in \Hom_R(R^{1/p^e},R)$ be the composition along the top row of this diagram, so that $y \chi =  \tau \circ \phi$.  By Lemma~\ref{makeitaniterateofafixedmap}, we can find $r \in R$ so that $z \phi (\blank) = \psi^e( r^{1/p^e} \cdot \blank)$.  We compute
\[
\tau(\psi^e((rx)^{1/p^e})) = \tau(z\phi(x^{1/p^e})) = z \tau(\phi(x^{1/p^e})) = yz\chi(x^{1/p^e}) = yz \not\in \m_A^{N'}
\]
and conclude $\psi^e((rx)^{1/p^e}) \not\in \m_A^{N'}R$ as $\tau$ is $A$-linear.  In particular, since $\m^N \subseteq \m_A^{N'}R$, we have $\psi^e((xR)^{1/p^e}) \not\subseteq \m^N$ as desired.
\end{proof}


In the proof of the main result of this section, we will need to use a well-known result on the stabilization of the images of an iterated $p^{-1}$-linear map. Rooted in the work of Hartshorne and Speiser \cite[Proposition 1.11]{HartshorneSpeiser} and later generalized by Lyubeznik  \cite{Lyubeznik}, the following version is due to Gabber \cite{Gabber-Dwork}.

\begin{Theorem}
\cite[Lemma~13.1]{Gabber-Dwork}
\cite[Proposition 2.14]{BlickleBockle}
\label{Hartshorne Speiser} Let $M$ be a finitely generated module over an $F$-finite ring $R$ and $0 \neq \vaprhi\in \Hom_R(M^{1/p},M)$. Then the descending chain of $R$-modules
\[
M \supseteq \phi (M^{1/p}) \supseteq \phi^2 (M^{1/p^2}) \supseteq \cdots \supseteq \phi^e(M^{1/p^e}) \supseteq \cdots
\]
stabilizes for $e \gg 0$.
\end{Theorem}

\begin{Theorem}\label{Aberbach Leuschke type Theorem 2} Let $(R,\m,k)$ be a complete local F-finite domain of dimension $d$ and $\{I_e\}_{e \in \N}$ a sequence of ideals so that $\m^{[p^e]}\subseteq I_e$ for all $e\in\N$. Suppose there exists a non-zero $\psi\in\Hom_R(R^{1/p},R)$ so that $\psi((I_{e+1})^{1/p})\subseteq I_e$ for all $e \in \N$. Then the limit 
$\lim_{e \to \infty} \frac{1}{p^{ed}}\ell_R(R/I_e)$
is positive if and only if $\bigcap_{e \in \N} I_e=0$.
\end{Theorem}

\begin{proof} By Lemma~\ref{Well known bound}, if $0 \neq c \in \bigcap_{e \in \N} I_e$, then $\ell_R(R/I_e) \leq \ell_R(R/(\m^{[p^e]},c)) \leq C(R/(c),\m) p^{e(d-1)}$ for all $e \in \N$.  It follows that $\bigcap_{e \in \N} I_e \neq 0$ implies $\lim_{e \to \infty} \frac{1}{p^{ed}}\ell_R(R/I_e) = 0$.

For the converse, suppose $\bigcap_{e \in \N} I_e = 0$. If $0 \neq \sigma = \psi^e(R^{1/p^e})$ for $e \gg 0$ is the stable image of $\psi$ as in Theorem~\ref{Hartshorne Speiser}, we also have $\bigcap_{e \in \N} (I_e : \sigma) = 0$ as $\sigma  \left( \bigcap_{e \in \N} (I_e : \sigma) \right) \subseteq \bigcap_{e \in \N} I_e = 0$ and $R$ is a domain.  In addition, it follows from $\psi(\sigma^{1/p}) = \sigma$ that $(I_{e+1} : \sigma) \subseteq (I_{e} : \sigma)$ for all $e \in \N$ as
\[
(I_{e+1} : \sigma) \sigma = (I_{e+1} : \sigma) \psi(\sigma^{1/p}) = \psi(((I_{e+1} : \sigma)^{[p]}\sigma)^{1/p}) \subseteq \psi(((I_{e+1} : \sigma)\sigma)^{1/p}) \subseteq \psi(I_{e+1}^{1/p}) \subseteq I_e.
\]
Take $N \in \N$ and $0 \neq c \in R$ as in Proposition~\ref{MainHHLemma}.  By Chevalley's Lemma \cite[Lemma~7]{ChevalleysLemma}, there is some $e_0 \in \N$ with $I_{e_0} \subseteq (I_{e_0} : \sigma) \subseteq \m^N$.  It follows from Proposition~\ref{MainHHLemma} that $I_{e+e_0} \subseteq (\m^{[p^e]}: c)$ for all $e \in \N$.  Indeed for each $x \in R \setminus (\m^{p^e}: c)$ we have $\psi^e((xR)^{1/p^e}) \not\subseteq \m^N$ and hence $\psi^e((xR)^{1/p^e}) \not\subseteq I_{e_0}$; as $\psi^e((I_{e + e_0})^{1/p^e}) \subseteq I_{e_0}$ we must have $x \not\in I_{e+e_0}$.  Using Lemma~\ref{colonsamelimit}, we compute
\[
\begin{array}{rcl}
\lim_{e \to \infty} \frac{1}{p^{(e+e_0)d}}\ell_R(R/I_{e + e_0}) &\geq& \frac{1}{p^{e_0 d}}\lim_{e \to \infty} \frac{1}{p^{ed}}\ell_R(R / (\m^{[p^e]}:c) ) \\ &=&
\frac{1}{p^{e_0 d}}\lim_{e \to \infty} \frac{1}{p^{ed}}\ell_R(R / \m^{[p^e]} ) \\ & = & \frac{1}{p^{e_0 d}} e_{HK}(R) \geq \frac{1}{p^{e_0 d}} > 0,
\end{array}
\]
which concludes the proof.
\end{proof}

Theorem \ref{Aberbach Leuschke type Theorem 2} is quite powerful and allows one to show positivity of the limits appearing in Theorem~\ref{MainExistenceTheoremIdealSequences} (ii.) in a number of different situations.  One should view this result not only as a generalization of the work done by Aberbach and Leuschke in \cite{AbLeu}, but also the work of Hochster and Huneke  \cite{HHJams}. If $R$ is a domain and $I \subseteq R$, recall that $x\in R$ is in the tight closure of $I^*$ of $I$ if there exists an element $0 \neq c \in R$ such that $c x^{p^e} \in I^{[p^e]}$ for all $e \in \N$.  The tight closure $I^*$ is an ideal containing $I$, and $I^{**} = I$.  A ring is said to be weakly $F$-regular if all ideals $I$ of $R$ are tightly closed, \textit{i.e.} satisfy $I^* = I$. See \cite{tightclosurebook} or \cite{fndtc} for further details.

\begin{Corollary} \cite[Theorem 8.17]{HHJams} \emph{(Length Criterion for Tight Closure).} \label{Length Criterion for Tight Closure} Let $(R,\m,k)$ be a complete local F-finite domain of dimension $d$. Suppose that $I \subseteq J$ is an inclusion of $\m$-primary ideals in $R$. Then $I^* = J^*$ if and only if $\e(I)=\e(J)$.
\end{Corollary}

\begin{proof}   
Applying the criterion to each term of a composition series of $J/I$, we may assume $J = (I,x)$ for some $x \in R$ with $(I:x) = \m$.  Consider the sequence of ideals $(I^{[p^e]}:x^{p^e})$ for each $e \in \N$.  We have that $(I^{[p^e]}:x^{p^e})^{[p]} \subseteq (I^{[p^{e+1}]}:x^{p^{e+1}})$ for all $e \geq 0$, so that in particular $\m^{[p^e]} \subseteq (I^{[p^e]}:x^{p^e})$ for each $e \in \N$.  Moreover, for any nonzero 
$\psi \in \Hom_R(R^{1/p},R)$ it is easy to check $\psi( (I^{[p^{e+1}]}:x^{p^{e+1}})^{1/p}) \subseteq (I^{[p^e]}:x^{p^e})$.  Since $(I,x)^{[p^e]} / I^{[p^e]} \simeq R / (I^{[p^e]}:x^{p^e})$, we see by Theorem~\ref{Aberbach Leuschke type Theorem 2}  that 
\begin{equation}
\label{relHKtocolon}
e_{HK}(I) - e_{HK}((I,x)) = \lim_{e \to \infty} \frac{1}{p^{ed}} \ell_R(R/(I^{[p^e]}:x^{p^e}))
\end{equation}
is zero if and only if $\bigcap_{e \in \N} (I^{[p^e]}: x^{p^e}) \neq 0$, which is equivalent to $x \in I^*$ by definition.
\end{proof}

Recall that there are also a number of well-known generalizations of tight closure and strong $F$-regularity \cite{HaraYoshida,HaraWatanabe,Takagi}.  If $R$ is an $F$-finite local domain, $\a$ is a nonzero ideal of $R$, and $t \in \R_{\geq 0}$, one can speak of the $\a^t$-tight closure $I^{*\a^t}$ of an ideal $I \subseteq R$.  By definition, if $x \in R$ then $x \in I^{*\a^t}$ if and only if there exists $0 \neq c \in R$ with $c\a^{\lceil t(p^e -1)\rceil}x^{p^e} \in I^{[p^e]}$ for all $e \in \N$.  The pair $(R,\a^t)$ is strongly $F$-regular provided for any $0 \neq x \in R$ there exists $e \in \N$ and $\phi \in \sC_e^{\a^t}$ with $\phi(x^{1/p^e}) = 1$. Similarly, if $R$ is an $F$-finite normal local domain and $\Delta$ is an effective $\Q$-divisor on $\Spec(R)$, we have the $\Delta$-tight closure $I^{*\Delta}$ of an ideal $I \subseteq R$ .  By definition, if $x \in R$ then $x \in I^{*\Delta}$ if and only if there exists $0 \neq c \in R$ with $c x^{p^e} \in I^{[p^e]}(R(\lceil (p^e - 1)\Delta \rceil))$ for all $e\in\N$.  
The pair $(R,\Delta)$ is strongly $F$-regular provided for any $0 \neq x \in R$ there exists $e \in \N$ and $\phi \in \sC^{(R,\Delta)}_e$ with $\phi(x^{1/p^e}) = 1$.
As of yet, there  does not  exist a well-formed theory of Hilbert-Kunz multiplicity for an ideal or divisor; nonetheless, one can use Theorem~\ref{Aberbach Leuschke type Theorem 2} to give analogues of  Theorem~\ref{Length Criterion for Tight Closure} for these operations.

\begin{Corollary}[Length Criterion for $\a^t$-tight closure]\label{Length Criterion for at-tight closure} Let $(R,\m,k)$ be a complete local \mbox{$F$-finite} domain of dimension $d$, $\a$ a nonzero ideal of $R$, and $t \in \R_{\geq 0}$. 
\begin{enumerate}[(i.)]
\item
For any $\m$-primary ideal $I$ and  $x \in R$, $\lim_{e \to \infty} \frac{1}{p^{ed}} \ell_R(R/(I^{[p^e]}: \a^{\lceil t(p^e - 1)\rceil} x^{p^e}))$ exists; moreover, this limit equals zero if and only if $x \in I^{*\a^t}$.
\item
The $F$-signature $s(R, \a^t) = \lim_{e \to \infty} \frac{1}{ p^{ed}} \ell_R(R / I_e^{\a^t})$ of $R$ along $\a^t$ is positive if and only if $(R, \a^t)$ is strongly $F$-regular.
\end{enumerate}
%
\end{Corollary}

\begin{proof} For (i.),
take $0 \neq \psi \in \sC^{\a^t}_1 = (\a^{\lceil t(p-1) \rceil})^{1/p}\cdot \Hom_R(R^{1/p},R)$, and consider the sequence of ideals $(I^{[p^e]}: \a^{\lceil t(p^e - 1)\rceil} x^{p^e})$ for $e \in \N$.  It is easy to check that $\psi((I^{[p^{e+1}]}: \a^{\lceil t(p^{e+1} - 1)\rceil} x^{p^{e+1}})^{1/p}) \subseteq (I^{[p^e]}: \a^{\lceil t(p^e - 1)\rceil} x^{p^e})$ for all $e \in \N$.  If $e_0 \in \N$ is sufficiently large that $\m^{[p^{e_0}]} \subseteq I$, then we have $\m^{[p^{e+e_0}]} \subseteq (I^{[p^e]}: \a^{\lceil t(p^e - 1)\rceil} x^{p^e})$ for all $e \in \N$.  After shifting the index, Theorem~\ref{MainExistenceTheoremIdealSequences} (ii.) and Theorem~\ref{Aberbach Leuschke type Theorem 2} apply, giving the existence of the limit and showing that it equals zero if and only if $\bigcap_{e \in \N} (I^{[p^e]}: \a^{\lceil t(p^e - 1)\rceil} x^{p^e}) \neq 0$, which is equivalent to $x \in I^{*\a^t}$ by definition. To see (ii.), simply note once more that $\psi((I_{e+1}^{\a^t})^{1/p}) \subseteq I_{e}^{\a^t}$ and apply Theorem~\ref{Aberbach Leuschke type Theorem 2}.
 \end{proof}

\begin{Corollary}[Length Criterion for $\Delta$-tight closure]\label{Length Criterion for Delta closure} Let $(R,\m,k)$ be a complete local \mbox{$F$-finite} normal domain of dimension $d$, and $\Delta$ an effective $\Q$-divisor on $\Spec(R)$. 
\begin{enumerate}[(i.)]
\item
For any $\m$-primary ideal $I \subseteq R$ and  $x \in R$, 
$\lim_{e \to \infty} \frac{1}{p^{ed}} \ell_R(R/(I^{[p^e]}R(\lceil (p^e - 1) \Delta \rceil): x^{p^e}))$ exists; moreover, this limit equals zero if and only if $x \in I^{*\Delta}$. 
\item
The $F$-signature $s(R,\Delta) = \lim_{e \to \infty} \frac{1}{p^{ed}} \ell_R(R/I_e^{(R,\Delta)})$ of $R$ along $\Delta$ is positive if and only if $(R,\Delta)$ is strongly $F$-regular.
\end{enumerate}
\end{Corollary}

\begin{proof} 
Let $0 \neq \psi \in \sC_1^\Delta = \im \left(\Hom_R((R(\lceil (p-1)\Delta\rceil))^{1/p},R) \to \Hom_R(R^{1/p},R)\right)$.  In other words, identifying $\psi$ with its unique extension to a $p^{-1}$-linear map on $K = \mathrm{Frac}(R)$, we have $\psi ((R(\lceil (p-1)\Delta\rceil))^{1/p}) \subseteq R$.  Twisting by $\lceil (p^e - 1 )\Delta \rceil$ and using that $p\lceil (p^e - 1 )\Delta \rceil + \lceil (p -1) \Delta \rceil \geq \lceil (p^{e+1} -1) \Delta \rceil$, observe that $\psi((R(\lceil (p^{e+1}-1)\Delta\rceil))^{1/p}) \subseteq R(\lceil (p^e-1)\Delta\rceil)$ for all $e \in \N$.

For (i.), consider the sequence of ideals $(I^{[p^e]}R(\lceil (p^e - 1) \Delta \rceil): x^{p^e})$ for $e \in \N$.  It is easy to check $\psi((I^{[p^{e+1}]}R(\lceil (p^{e+1} - 1) \Delta \rceil): x^{p^{e+1}})^{1/p})\subseteq (I^{[p^e]}R(\lceil (p^e - 1) \Delta \rceil): x^{p^e})$ for all $e \in \N$.  If $e_0 \in \N$ is sufficiently large that $\m^{[p^{e_0}]} \subseteq I$, then we have $\m^{[p^{e+e_0}]} \subseteq (I^{[p^e]}R(\lceil (p^e - 1) \Delta \rceil): x^{p^e})$ for all $e \in \N$.  After shifting the index, Theorem~\ref{MainExistenceTheoremIdealSequences} (ii.) and Theorem~\ref{Aberbach Leuschke type Theorem 2} apply, giving the existence of the limit and showing that it equals zero if and only if $\bigcap_{e \in \N} (I^{[p^e]}R(\lceil (p^e - 1) \Delta \rceil): x^{p^e}) \neq 0$, which is equivalent to $x \in I^{*\Delta}$ by definition.  For (ii.), simply note once more that $\psi((I_{e+1}^{(R,\Delta)})^{1/p}) \subseteq I_e^{(R,\Delta)}$ and apply Theorem~\ref{Aberbach Leuschke type Theorem 2}.
\end{proof}

\begin{Remark}    Note that the analogue of Theorem~\ref{Aberbach Leuschke type Theorem 2} for sequences of ideals satisfying Theorem~\ref{MainExistenceTheoremIdealSequences} (i.) instead of (ii.) is false. Here is an easy counterexample.  Let $(R,\m,k)$ be a complete regular local ring of dimension $d$, and consider the sequence of ideals $I_e = \m^{[p^{\lfloor e / 2\rfloor}]}$ for $e \in \N$.  Then $\m^{[p^e]} \subseteq I_e$ and $I_e^{[p]} \subseteq I_{e+1}$ for all $e \in \N$ and certainly $\bigcap_{e \in \N} I_e = 0$.  However, we have $\lim_{e \to \infty} \frac{1}{p^{ed}}\ell_R(R/I_e) = \lim_{e \to \infty} \frac{1}{p^{(e - \lfloor e/2 \rfloor)d}} =0 $. 
\end{Remark}

\renewcommand{\P}{\mathcal{P}}
To see yet another application of Theorem~\ref{Aberbach Leuschke type Theorem 2}, suppose that $(R, \m, k)$ is an $F$-split $F$-finite local ring.  If $I_e^{\mathrm{F-sig}} = ( r \in R \mid \phi(r^{1/p^e}) \in \m \mbox{ for all } \phi \in \Hom_R(R^{1/p^e},R) )$, it is straightforward to check that $\P = \bigcap_{e \in \N} I_e^{\mathrm{F-sig}}$ is a prime ideal (see \cite[Lemma~4.7]{Tucker2012}), coined the $F$-splitting prime by  Aberbach and Enescu in \cite{AberbachEnescu}.  Furthermore, they suspected that the dimension of the $F$-splitting prime governed the growth rate of $\frk_R(R^{1/p^e})$ when $R$ is not strongly $F$-regular; this observation was verified in joint work of the second author with Blickle and Schwede. 

\begin{Theorem}
\cite{Fsigpairs1}
\label{Fsplitratio}
Suppose $(R, \m, k)$ is an $F$-split $F$-finite local ring and $\P$ is the $F$-splitting prime.  If $n = \dim(R/\P)$, the limit $r_F(R) = \lim_{e \to \infty} \frac{1}{[k^{1/p^e}:k] \cdot p^{en}} \frk_R(R^{1/p^e})$ exists and is positive, called the $F$-splitting ratio.  Moreover, we have that $\frac{1}{[k^{1/p^e}:k]} \frk_R(R^{1/p^e}) = r_F(R) p^{en} + O(p^{e(n-1)})$.
\end{Theorem}

\begin{proof} Without loss of generality, we may assume $R$ is complete. If $I_e^{\mathrm{F-sig}} = \langle r \in R \, | \, \phi(r^{1/p^e}) \in \m \mbox{ for all } \phi \in \Hom_R(R^{1/p^e},R) \rangle$, we have $\m^{[p^e]} \subseteq I_e^{\mathrm{F-sig}}$ and $(I_e^{\mathrm{F-sig}})^{[p]} \subseteq I_{e+1}^{\mathrm{F-sig}}$ for all $e \in \N$.  Fixing a surjective map $\psi \in \Hom_R(R^{1/p},R)$, we have $\psi((I_{e+1}^{\mathrm{F-sig}})^{1/p}) \subseteq I_e^{\mathrm{F-sig}}$.  In particular, this implies $\psi(\P^{1/p}) \subseteq \P$ and so $\psi$ induces a map $\overline{\psi} \in \Hom_{\overline{R}}(\overline{R}^{1/p},\overline{R})$ where $\overline{R} = R/\P$.  Note that $\overline{\psi}$ is still surjective and hence nonzero.
 Thus, passing to the sequence of ideals $ \overline{I_e} = I_e^{\mathrm{F-sig}} \overline{R}$, we have $\m^{[p^e]}\overline{R} \subseteq \overline{I_e}$, $\overline{I_e}^{[p]} \subseteq \overline{I_{e+1}}$, and $\overline{\psi}((\overline{I_{e+1}})^{1/p}) \subseteq \overline{I_e}$ for all $e \in \N$.  Moreover, $\bigcap_{e \in \N}\overline{I_e} = 0$ in $\overline{R}$.  The result now 
 follows immediately from Corollary~\ref{Limits exist combined}  together with Theorem~\ref{Aberbach Leuschke type Theorem 2}, using that $\frac{\frk_R(R^{1/p^e})}{[k^{1/p^e}:k]} = \ell_R(R/I_e^{\mathrm{F-sig}} ) = \ell_{\overline{R}}(\overline{R} / \overline{I_e})$ for all $e \in \N$.
\end{proof}

\begin{Remark}
 It is straightforward to generalize the notions of $F$-splitting prime and $F$-splitting ratio to arbitrary Cartier subalgebras; see \cite{Fsigpairs1} for further details.  In all cases, the method of the proof of Theorem~\ref{Fsplitratio} applies and greatly simplifies the proofs.  In particular, the methods of the proofs of Theorem~\ref{Dsignatureexists} and Theorem~\ref{Aberbach Leuschke type Theorem 2} immediately give an alternative proof of the following result.
\end{Remark}

\begin{Corollary}
 \label{deefsigpositivity}
 \cite[Theorem 3.18]{Fsigpairs1}
 Let $(R, \m, k)$ be an $F$-finite local domain of dimension $d$ and $\sD$ a Cartier subalgebra on $R$.  Then the $F$-signature  
$s(R,\sD) = \lim_{ \substack{e \to \infty \\ e \in \Gamma_\sD}} \frac{1}{[k^{1/p^e}:k] \cdot p^{ed}} a_e^\sD$ of $R$ along $\sD$ is positive if and only if $(R,\sD)$ is strongly $F$-regular.
\end{Corollary}

\begin{Remark}
Another straightforward generalization comes modifying the condition in Theorem~\ref{MainExistenceTheoremIdealSequences}~(ii) and Theorem~\ref{Aberbach Leuschke type Theorem 2} to consider sequences of $\m$-primary ideals governed by a non-zero $p^{-e_0}$-linear map for some $e_0 \in \N$. Beginning with short exact sequences in Lemma~\ref{Map Lemma} replacing $R^{1/p}$ with $R^{1/p^{e_0}}$ and tracing through the arguments of previous sections, one readily arrives at the following statement.
\end{Remark}

\begin{Corollary}
Let $(R,\m,k)$ be a complete local F-finite domain of dimension $d$ and $e_0 \in \N$.  Suppose $\{I_{n} \}_{n \in \N}$ a sequence of ideals so that $\m^{[p^{ne_0}]}\subseteq I_{n}$ for all $n\in\N$. Suppose there exists a non-zero $\psi\in\Hom_R(R^{1/p^{e_0}},R)$ so that $\psi((I_{n+1})^{1/p^{e_0}})\subseteq I_n$ for all $n \in \N$. Then the limit 
$\lim_{n \to \infty} \frac{1}{p^{ne_0 d}}\ell_R(R/I_n)$
exists and is positive if and only if $\bigcap_{n \in \N} I_n=0$.
\end{Corollary}

\section{F-signature and Minimal Relative Hilbert-Kunz Multiplicity}\label{Section F-signature and Minimal Relative Hilbert-Kunz Multiplicity}

Our next aim is to realize the $F$-signature as the infimum of relitive differences in the Hilbert-Kunz multiplicities of the cofinite ideals in a local ring (Corollary \ref{WY Type result corollary}).    
After first bounding such differences from below by the $F$-signature (Lemma \ref{Iecontainscolons}), we will make use of approximately Gorenstein sequences to find differences arbitrarily close to the $F$-signature. The crucial step and main technical result is Theorem \ref{WY Type result}, which uses the uniformity of the constants tracked above swap limits between  iterations of Frobenius and progression in an approximately Gorenstein sequence.

The remainder of the section is reserved for constructions of explicit sequences of relative Hilbert-Kunz differences that approach the $F$-signature (see Corollary \ref{reducex1}); we also analyze when the infimum is known to be achieved (Corollary \ref{QGorAnalysis}).  Generalizations to divisor and ideal pairs are also given (Corollaries \ref{a^t signature and WY result}, \ref{Delta signature and WY result}).

\begin{Lemma}
\label{Iecontainscolons}
Suppose that $(R, \m, k)$ is an $F$-finite local ring and $I_e^{\mathrm{F-sig}} = ( r \in R \mid \phi(r^{1/p^e}) \in \m \mbox{ for all } \phi \in \Hom_R(R^{1/p^e},R) )$ for $e \in \N$. Then
\begin{equation}
\label{Iesumofcolons}
    I_e^{\mathrm{F-sig}} \supseteq \sum_{ \substack{I \subseteq J \subseteq R \\  0 < \ell_R(J/I) < \infty}} (I^{[p^e]}:J^{[p^e]}) \supseteq \sum_{\substack{I \subseteq R, \; \ell_R(R/I) < \infty \\  x \in R, \; ( I : x )=\m}} (I^{[p^e]} : x^{p^e})
\end{equation}
and we have
\begin{equation}
\label{relativeHKsignatureeasyinequality}
s(R) \leq \inf_{\substack{I \subseteq J  \subseteq R, \;\ell_R(R/I) < \infty  \\ I \neq J, \; \ell_R(R/J) < \infty}} \frac{e_{HK}(I) - e_{HK}(J)}{\ell_R(J/I)} \leq \inf_{\substack{I \subseteq R, \; \ell_R(R/I) < \infty \\  x \in R, \; ( I : x )=\m}} e_{HK}(I) - e_{HK}((I,x)).
\end{equation}
\end{Lemma}

\begin{proof}
 If $I \subsetneq J$ is a proper inclusion of ideals with $\ell_R(J/I) < \infty$ and $\phi \in \Hom_R(R^{1/p^e},R)$, we have that $\phi((I^{[p^e]}:J^{[p^e]})^{1/p^e}) \subseteq (I:J) \subseteq \m$.  It follows that $\sum_{\substack{I \subseteq J \subseteq R \\  0 < \ell_R(J/I) < \infty}} (I^{[p^e]}:J^{[p^e]}) \subseteq I_e^{\mathrm{F-sig}}$. Since the sum on the  right is over a smaller set of proper inclusions,  \eqref{Iesumofcolons}  follows immediately.
 
 For \eqref{relativeHKsignatureeasyinequality}, if $I \subseteq R$ is $\m$-primary and $x \in R$ with $(I:x) = \m$, then $(I^{[p^e]}:x^{p^e}) \subseteq   I_e^{\mathrm{F-sig}}$ for all $e \in \N$ implies $e_{HK}(I) - e_{HK}((I,x)) \geq s(R)$ using \eqref{relHKtocolon}.  Moreover, if $I \subsetneq J$ is a proper inclusion of $\m$-primary ideals, summing up this inequality for each  factor in a composition series of $J/I$ shows ${\ell_R(J/I)} s(R) \leq (e_{HK}(I) - e_{HK}(J))$.  The inequalities in \eqref{relativeHKsignatureeasyinequality} now follow immediately, noting again that the infimum on the right is over a smaller set of proper inclusions.
 \end{proof}
 
Recall that a local ring $(R, \m, k)$ is said to be approximately Gorenstein if there exists a descending chain of irreducible ideals $\{J_t\}_{t \in \N}$ cofinal with powers of the maximal ideal. In particular, each $R / J_t$ is a zero-dimensional Gorenstein local ring and has a one dimensional socle.  By \cite[Theorem~1.6]{HochsterPurity}, a reduced excellent local ring is always approximately Gorenstein. It is easy to check that equality holds throughout \eqref{Iesumofcolons} for such rings.
 
\begin{Lemma}
\label{approxgorcolons}
Suppose that $(R, \m, k)$ is an approximately Gorenstein $F$-finite local ring, and $\{J_t\}_{t \in \N}$ is a descending chain of irreducible ideals cofinal with the powers of $\m$.  If $\delta_t \in R$  generates the socle of $R/J_t$, then for all $e \in \N$ we have $ (J_t^{[p^e]}:\delta_t^{[p^e]}) \subseteq (J_{t+1}^{[p^e]}:\delta_{t+1}^{[p^e]})$ and 
\[
\begin{array}{rcl}
I_e^{\mathrm{F-sig}} &=& \langle r \in R \, | \, \phi(r^{1/p^e}) \in \m \mbox{ for all } \phi \in \Hom_R(R^{1/p^e},R) \rangle \\
& = & \sum_{t \in \N} (J_t^{[p^e]}:\delta_t^{[p^e]}) \\
& = & \bigcup_{t \in \N} (J_t^{[p^e]}:\delta_t^{[p^e]}) \\
& = & (J_{t_e}^{[p^e]}:\delta_{t_e}^{[p^e]}) \mbox{ for all $t_{e} \gg 0$ sufficiently large.}
\end{array}
\]
Moreover, $R$ is weakly $F$-regular if and only if $J_t^* = J_t$ is tightly closed for all $t \in \N$.
\end{Lemma}

\begin{proof} Since each $R/J_t$ is an Artinian Gorenstein local ring, we have that $\Ann_{E}(J_t) \simeq R/J_t$ where $E = E_R(k)$
is the injective hull of the residue field of $R$.  Thus, we may view $E = E_R(k) =\varinjlim R/J_t$ as the direct limit of inclusions $R/J_t \to R/J_{t+1}$ mapping (the class of) $\delta_t \mapsto \delta_{t+1}$.  In particular, after applying $\blank \otimes_R R^{1/p^e}$, one sees that $ (J_t^{[p^e]}:\delta_t^{[p^e]}) \subseteq (J_{t+1}^{[p^e]}:\delta_{t+1}^{[p^e]}) $ for all $t \in \N$.  It follows immediately that $\sum_{t \in \N} (J_t^{[p^e]}:\delta_t^{[p^e]}) 
= \bigcup_{t \in \N} (J_t^{[p^e]}:\delta_t^{[p^e]}) 
= (J_{t_e}^{[p^e]}:\delta_{t_e}^{[p^e]}) $ for all $t_e \gg 0$.

An inclusion $R \to M$ to a finitely generated $R$-module $M$ determined by $1 \mapsto m$ splits if and only if $E \to E \otimes_R M$ remains injective, which is equivalent to $\delta_t m \not\in J_t M$ for all $t \in \N$.  See \cite[page~155]{fndtc} for further details. In particular, if $x \in R$, applying this splitting criterion to the map $R \to R^{1/p^e}$ with $1 \mapsto x^{1/p^e}$ gives that $x \in R \setminus I_e^{\mathrm{F-sig}}$ if and only if $x \in R \setminus (J_t^{[p^e]}:\delta_t^{[p^e]})$ for all $t \in \N$, and so we have that $I_e^{\mathrm{F-sig}} = (J_{t_e}^{[p^e]}:\delta_{t_e}^{[p^e]})$ for $t_e \gg 0$.

Lastly, suppose there is an ideal $I \subseteq R$ that is not tightly closed. Then $I=\bigcap_{n\in\N}(I+\m^n)$ is an intersection of $\m$-primary ideals. The arbitrary intersection of tightly closed ideals is tightly closed, \cite[Proposition~4.1(b)]{HHJams}. Hence  we may replace $I$ with an $\m$-primary ideal which is not tightly closed and choose  $x \in I^*$ with $\m = (I:x)$. Since $R/I$ injects into a direct sum of copies of $E$, we can find an $R$-module homomorphism $R/I \to E$ so that  $x + I$ has nonzero image in $E$ and hence must generate the socle $k \subseteq E$.  Using that $E = \varinjlim R/J_t$, we may assume $R/I \to R/J_t$ and $x + I \mapsto \delta_t + J_t$ for some $t \in \N$.  For each $e \in \N$, applying $\blank \otimes_R R^{1/p^e}$ and viewing as a $R^{1/p^e}$-module gives $R/I^{[p^e]} \to R/J_t^{[p^e]}$ where $x^{p^e} + I^{[p^e]} \mapsto \delta_t^{p^e} + J_t^{[p^e]}$.  In particular, it follows that $(I^{[p^e]}: x^{p^e}) \subseteq (J_t^{[p^e]}: \delta_t^{[p^e]})$.  
Thus $ \bigcap_{e \in \N} (I^{[p^e]} : x^{p^e}) \subseteq \bigcap_{e \in \N} (J_t^{[p^e]} : \delta_t^{p^e})$  is not contained in any minimal prime of $R$; it follows that $\delta_t \in J_t^*$ and hence $J_t$ is not tightly closed (\textit{cf.} \cite[Proposition~8.23(f)]{HHJams}).
\end{proof}

In the next result, we show how to make use of the uniformity of constants from Theorem~\ref{MainExistenceTheoremIdealSequences} together with approximately Gorenstein sequences in order to compare $F$-signature and Hilbert-Kunz multiplicity.  As an application, we answer a question posed by Watanabe and Yoshida \cite[Question~1.10]{WY}.

\begin{Theorem}\label{WY Type result} Let $(R,\m,k)$ be a complete local F-finite domain of dimension $d$, and fix $0 \neq c \in R$.  Suppose we are given sequences of ideals $\{ I_{t,e} \}_{t,e \in \N}$ satisfying 
 $\m^{[p^e]} \subseteq I_{t,e}$,
 $c (I_{t,e}^{[p]}) \subseteq I_{t,e+1}$, and
 $I_{t,e} \subseteq I_{t+1,e}$ 
for all $t, e \in \N$.
Then
 \[
  \lim_{e\rightarrow \infty}\lim_{t\rightarrow \infty}\frac{1}{p^{ed}}\ell_R(R/I_{t,e}) = \lim_{e \to \infty} \frac{1}{p^{ed}} \ell_R(R/I_e) =\lim_{t\rightarrow \infty}\lim_{e\rightarrow \infty}\frac{1}{p^{ed}}\ell_R(R/I_{t,e})
 \]
 where $I_e = \sum_{t \in \N} I_{t,e}$.
\end{Theorem}

\begin{proof} For each fixed $t \in \N$, as $\m^{[p^e]} \subseteq I_{t,e}$ and $c(I_{t,e}^{[p]}) \subseteq I_{t,e+1}$ for $e \in \N$, Theorem~\ref{MainExistenceTheoremIdealSequences}~(i) guarantees the existence of $\eta_t = \lim_{e \to \infty} \frac{1}{p^{ed}} \ell_R(R/I_{t,e})$ and provides a uniform positive constant $C \in \R$ so that
\begin{equation}
\label{unifboundinlimitswitch1}
\eta_t \leq \frac{1}{p^{ed}} \ell_R(R/I_{t,e}) + \frac{C}{p^e}
\end{equation}
for all $t,e \in \N$.  The sequence $\{ I_e \}_{e \in \N}$ inherits the properties $\m^{[p^e]} \subseteq I_e$ and $c(I_e^{[p]}) \subseteq I_{e+1}$ for $e \in \N$, hence $\lim_{e \to \infty} \frac{1}{p^{ed}} \ell_R(R/I_e)$  exists as well.  Since $I_{t,e} \subseteq I_e$ and so $  \ell_R(R/I_e) \leq \ell_R(R/ I_{t,e})$ for all $t,e \in \N$, applying $\lim_{e \to \infty}$ gives 
\begin{equation}
\label{unifboundinlimitswitch2}
\lim_{e \to \infty} \frac{1}{p^{ed}} \ell_R(R/I_{e}) \leq \eta_t  
\end{equation}
 for all $t \in \N$.

Since 
$I_{t,e} \subseteq I_{t+1,e}$ is increasing in $t$ for fixed $e$, it follows that $\eta_t \geq \eta_{t+1} \geq 0$ for all $t\in \N$ and hence $\lim_{t \to \infty} \eta_t$ exists.  We also have 
$I_e = I_{t_e,e}$ for $t_e \gg 0$, so that $\lim_{e\rightarrow \infty}\lim_{t\rightarrow \infty}\frac{1}{p^{ed}}\ell_R(R/I_{t,e}) = \lim_{e \to \infty} \frac{1}{p^{ed}} \ell_R(R/I_e) $.  Applying $\lim_{t \to \infty}$ to \eqref{unifboundinlimitswitch1} and \eqref{unifboundinlimitswitch2} gives
\[
\lim_{e \to \infty} \frac{1}{p^{ed}} \ell_R(R/I_{e}) \leq \lim_{t \to \infty}\eta_t \leq \frac{1}{p^{ed}} \ell_R(R/I_{e}) + \frac{C}{p^e}
\]
for all $e \in \N$.  Further taking $\lim_{e \to \infty}$ gives $\lim_{t \to \infty}\eta_t = \lim_{e \to \infty} \frac{1}{p^{ed}} \ell_R(R/I_{e})$ and completes the proof.  
\end{proof}

\begin{Theorem}\label{Corollary to WY result} Let $(R,\m,k)$ be an approximately Gorenstein $F$-finite local ring of dimension~$d$.  Suppose $\{J_t\}_{t \in \N}$ is a descending chain of irreducible ideals cofinal with the powers of $\m$, and $\delta_t \in R$  generates the socle of $R/J_t$.  Then $s(R) = \lim_{t \to \infty} e_{HK}(J_t) - e_{HK}((J_t, \delta_t))$.
\end{Theorem} 

\begin{proof}  
Both invariants are unchanged after completion, so we may assume $R$ is complete. Suppose first that $R$ is not weakly $F$-regular, 
so that $\delta_t \in J_t^*$ for some $t \in \N$ and $(J_t^{[p^e]}:\delta_t^{[p^e]}) \subseteq (J_{t+1}^{[p^e]}:\delta_{t+1}^{[p^e]}) \subseteq \cdots \subseteq (J_{t+t'}^{[p^e]}:\delta_{t+t'}^{[p^e]}) \subseteq \cdots $  for all $e,t' \in \N$  by Lemma~\ref{approxgorcolons}. If $c \in \bigcap_{e \in \N} (J_t^{[p^e]} : \delta_t^{p^e}) \subseteq \bigcap_{e \in \N} (J_{t+t'}^{[p^e]} : \delta_{t+t'}^{p^e})$ is not in any minimal prime, then we have that $0 \leq e_{HK}(J_{t+t'}) - e_{HK}((J_{t+t'},\delta_{t+t'})) \leq \lim_{e \to \infty} \frac{1}{p^{ed}} \ell_R(R/(c, \m^{[p^e]})) = 0$ using \eqref{relHKtocolon} from Corollary~\ref{Length Criterion for Tight Closure} and applying Lemma~\ref{Well known bound} with $M = R/(c)$. Using Lemma~\ref{Iecontainscolons}, we have that $s(R) =  \lim_{t' \to \infty} e_{HK}(J_{t+t'}) - e_{HK}((J_{t+t'}, \delta_{t+t'})) = 0$ as desired.  

Thus, we assume for the remainder that $R$ is weakly $F$-regular and hence a domain.  Consider the sequences of ideals $I_{t,e} = (J_t^{[p^e]}:\delta_t^{[p^e]})$ for $t,e \in \N$.  We check 
\[
\begin{array}{c}
\m^{[p^e]} = (J_t : \delta_t)^{[p^e]} \subseteq (J_t^{[p^e]}:\delta_t^{[p^e]}) = I_{t,e} 
\\
I_{t,e}^{[p]} = (J_t : \delta_t)^{[p^e]} \subseteq (J_t^{[p^e]}:\delta_t^{[p^e]})^{[p]} \subseteq (J_t^{[p^{e+1}]}:\delta_t^{[p^{e+1}]}) = I_{t,e+1} \\
I_{t,e} = (J_t^{[p^e]}:\delta_t^{[p^e]}) \subseteq (J_{t+1}^{[p^e]}:\delta_{t+1}^{[p^e]}) = I_{t+1,e}
\end{array}
\]
so that Theorem~\ref{WY Type result} applies with $c = 1$.  Using Lemma~\ref{approxgorcolons} and \eqref{relHKtocolon} from Corollary~\ref{Length Criterion for Tight Closure}, we conclude
\[
\begin{array}{ccccc}
s(R) &=&\displaystyle \lim_{e \to \infty} \frac{1}{p^{ed}}\ell_R(R/I_e^{\mathrm{F-sig}})
&=& \displaystyle \lim_{e \to \infty} \lim_{t \to \infty} \frac{1}{p^{ed}}\ell_R(R/I_{t,e}) \\
&=& \displaystyle \lim_{t \to \infty} \lim_{e \to \infty} \frac{1}{p^{ed}}\ell_R(R/I_{t,e}) &=& \displaystyle \lim_{t \to \infty} e_{HK}(J_t) - e_{HK}((J_t, \delta_t))
\end{array}
\]
which completes the proof.
\end{proof}

Theorem \ref{Corollary to WY result} provides a positive answer to \cite[Question~1.10]{WY}.

\begin{Corollary}\label{WY Type result corollary}
 If $(R,\m,k)$ is an F-finite local ring, then
 \[
 s(R) = \inf_{\substack{I \subseteq J  \subseteq R, \;\ell_R(R/I) < \infty  \\ I \neq J, \; \ell_R(R/J) < \infty}} \frac{e_{HK}(I) - e_{HK}(J)}{\ell_R(J/I)} = \inf_{\substack{I \subseteq R, \; \ell_R(R/I) < \infty \\  x \in R, \; ( I : x )=\m}} e_{HK}(I) - e_{HK}((I,x)).
 \]
\end{Corollary}

\begin{proof}
If $R$ is not reduced, we can find some $0 \neq x \in R$ with $x^{p} = 0$.  For $n \gg 0$, we have $x \not\in \m^n$ and can find an ideal $I$ with $\m^n \subseteq I \subseteq (\m^n,x)$ where $(I:x) = \m$.  Since $x^{p^e} = 0$ for all $e \in \N$ we have $I^{[p^e]} = (I,x)^{[p^e]}$, and thus $e_{HK}(I) = e_{HK}((I,x))$.  It follows from Lemma~\ref{Iecontainscolons} that $s(R) = 0$ and equality holds throughout \eqref{relativeHKsignatureeasyinequality}. Thus, we may assume $R$ is reduced.  By \cite[Theorem~1.7]{HochsterPurity}, $R$ is approximately Gorenstein and Theorem~\ref{Corollary to WY result} implies equality holds throughout \eqref{relativeHKsignatureeasyinequality} as desired.
\end{proof}

Note that, when $(R,\m,k)$ is a complete Cohen-Macaulay local $F$-finite domain of dimension $d$, one can make Theorem~\ref{Corollary to WY result} more explicit still.  Recall that a canonical ideal $J \subseteq R$ is an ideal such that $J$ is isomorphic to a canonical module $\omega_R$, which exists as $R$ is assumed complete.   A canonical ideal $J$ is necessarily unmixed with height $1$, and moreover
$R/J$ will be Gorenstein of dimension $d -1$, see \cite[Proposition~3.3.18]{BrunsHerzog}.  When $R$ is also normal, fixing a canonical ideal is equivalent to fixing a choice of  effective anticanonical divisor. 

\begin{Corollary}
\label{reducex1}
 Suppose that $(R,\m,k)$ is a complete Cohen-Macaulay local $F$-finite domain of dimension $d$, and $J$ is a canonical ideal of $R$.  Let $x_1 \in J$ and $x_2, \ldots, x_d \in R$ be chosen so that $x_1, \ldots, x_d$ give a system of parameters for $R$, and suppose $\delta \in R$ generates the socle of $R/(J, x_2, \ldots, x_d)$.  Then
\[
 s(R) = \lim_{t \to \infty} e_{HK}((J, x_2^t, \ldots, x_d^t)) - e_{HK}((J, x_2^t, \ldots, x_d^t, (x_2^{t-1}\cdots x_d^{t-1} \delta))).
 \]
\end{Corollary}

\begin{proof}
Note that $x_1^{t-1}J$ is yet another canonical ideal for any $t \in \N$, and $x_2^t, \ldots, x_d^t$ give a system of parameters for $R/(x_1^{t-1}J)$.  Thus, the sequence $J_t = (x_1^{t-1}J, x_2^t, \ldots, x_d^t)$ gives a descending chain of irreducible ideals cofinal with the powers of $\m$.  It is easy to check that $\delta_t = x_1^{t-1}x_2^{t-1} \cdots x_d^{t-1}\delta$ generates the socle of $R/J_t$ using that $x_1, \ldots, x_d$ form a regular sequence, which further implies
$
(J_t^{p^e}:\delta_t^{p^e}) = ((J, x_2^{t}, \ldots, x_d^{t})^{[p^e]} : (x_2^{t-1} \cdots x_d^{t-1}\delta)^{p^e})
$
for all $t,e \in \N$.  Thus, Theorem~\ref{Corollary to WY result} gives
\begin{eqnarray*}
s(R) 
&=& \lim_{t \to \infty} \frac{1}{p^{ed}} \ell_R(R/(J_t^{p^e}:\delta_t^{p^e})) \\ 
& = &  \lim_{t \to \infty} e_{HK}((x_1^{t-1}J, x_2^t, \ldots, x_d^t)) - e_{HK}((x_1^{t-1}J, x_2^t, \ldots, x_d^t, (x_1^{t-1}x_2^{t-1}\cdots x_d^{t-1} \delta))) \\ 
&=& \lim_{t \to \infty} \frac{1}{p^{ed}} \ell_R(R/((J, x_2^{t}, \ldots, x_d^{t})^{[p^e]} : (x_2^{t-1} \cdots x_d^{t-1}\delta)^{p^e})) \\
& = &  \lim_{t \to \infty} e_{HK}((J, x_2^t, \ldots, x_d^t)) - e_{HK}((J, x_2^t, \ldots, x_d^t, (x_2^{t-1}\cdots x_d^{t-1} \delta)))
\end{eqnarray*}
using the relation in \eqref{relHKtocolon} once more.
\end{proof}

One can push the above analysis further still. In the notation of the previous proof, when $R$ is normal $x_2$ can be chosen so that $R_{x_2}$ is Gorenstein and $J_{x_2}$ is principal.  This allows one to remove the exponent $t$ on $x_2$ in the limit above using (i.) from the subsequent lemma.  Following the methods of \cite{AberbachConditionsforWeakStrong} (\textit{cf.} \cite{F-sig?, MacCrimmonThesis,Yao}), we present a complete treatment in Corollary~\ref{QGorAnalysis} below.

\begin{Lemma}
\label{colonslemma}
Suppose that $(R,\m,k)$ is a complete Cohen-Macaulay local $F$-finite normal domain of dimension $d$, and $D$ an effective Weil divisor on $\Spec(R)$.  Put $J = R(-D) \subseteq R$ so that $J^{(n)} = R(-nD)$ for $n \in \N$.  Let $x_1 \in J$ and $x_2, \ldots, x_d \in R$ be chosen so that $x_1, \ldots, x_d$ give a system of parameters for $R$, and fix $e \in \N$.
\begin{enumerate}[(i.)]
\item
If $x_2 J \subseteq a_2 R$ for some $a_2 \in J$, there exists $b_2 \in J$ so that $a_2, x_2 + b_2, x_3, \ldots, x_d$ give a system of parameters for $R$.  Moreover, for any non-negative integers $N_2, \ldots, N_d$ with $N_2 \geq 2$, we have that
\[
\begin{array}{ll}
 & ((J^{(p^{e})},x_2^{N_2 p^e}, x_3^{N_3 p^e}, \ldots, x_d^{N_d p^e}):x_2^{(N_2 - 1)p^e} ) \\ = &
 ((J^{[p^{e}]},x_2^{N_2 p^e}, x_3^{N_3 p^e}, \ldots, x_d^{N_d p^e}):x_2^{(N_2 - 1)p^e} ) \\ =  & ((J^{[p^{e}]},x_2^{2 p^e}, x_3^{N_3 p^e}, \ldots, x_d^{N_d p^e}):x_2^{p^e} ).
\end{array}
\]
\item
Suppose $x_d^n J^{(n)} \subseteq a_d R$ for some $n \in \N$ and $a_d \in J^{(n)}$. Then there exists $b_d \in J$ so that $a_d,  x_2, \ldots, x_{d-1}, x_d + b_d$ give a system of parameters for $R$.  Moreover, for any non-negative integers $N_2, \ldots, N_d$ with $N_d \geq 2$, we have that
\begin{equation*}
\begin{array}{ll}
& ((J^{(p^{e})},x_2^{N_2 p^e},\ldots, x_{d-1}^{N_{d-1} p^e},  x_d^{N_d p^e}):  x_d^{(N_d -1) p^e}) \\  \subseteq 
& ((J^{(p^{e})},x_2^{N_2 p^e},\ldots, x_{d-1}^{N_{d-1} p^e},  x_d^{2 p^e}): x_1^{n} x_d^{p^e})
\end{array}
\end{equation*}

\end{enumerate}
\end{Lemma}

\begin{proof}
For (i.), note first that $a_2$ is a non-zero divisor on $R / (x_3, \ldots, x_d)$ as its multiple $x_1x_2 \in x_2J$ is  a non-zero divisor. Hence $\dim(R / (x_3, \ldots, x_d, a_2)) = d -1$.
Since $\ell_R(R / (J, x_2, \ldots, x_d)) \leq \ell_R(R/(x_1, \ldots, x_d)) < \infty$, we know that $(x_2, J)$ is not contained in any minimal prime ideal of $(x_3, \ldots, x_d, a_2)$.  In particular, using standard prime avoidance arguments, one can find $b_2 \in J$ so that $x_2 + b_2$ also avoids all of the minimal primes of $(x_3, \ldots, x_d, a_2)$;
it follows that the sequence $x_3, \ldots, x_d, a_2, x_2 + b_2$ is a system of parameters for $R$.

Since $R(-\divisor(x_2) - D) = x_2 J \subseteq a_2 R = R(-\divisor(a_2))$, we have that $D + \divisor(x_2) \geq \divisor(a_2)$ and hence also $p^eD + \divisor(x_2^{p^e}) \geq \divisor(a_2^{p^e})$; it follows that $x_2^{p^e}J^{(p^e)} \subseteq a_2^{p^e}R$.
The ideals inclusions $\supseteq$ in (i.) are clear, so we need only check the reverse.  
Let $I = (x_3^{N_3 p^e}, \ldots, x_d^{N_d p^e})$ and suppose that we have $cx_2^{(N_2 -1) p^e} \in (J^{(p^e)},x_2^{N_2 p^e},I)$ for some $c \in R$.
We can find $r_2 \in R$ with $(c - r_2x_2^{p^e})x_2^{(N_2 - 1)p^e} \in (J^{(p^{e})},I) $.
Since $b_2 \in J$, we have $b^{p^e} \in J^{(p^e)}$ and so $(c - r_2x_2^{p^e})(x_2 + b_2)^{(N_2 - 1)p^e} \in (J^{(p^{e})},  I) $.  
Multiplying through by $x_2^{p^e}$ and using that $x_2^{p^e} J^{(p^e)} \subseteq a_2^{p^e}$ gives $(cx_2^{p^e} - r_2x_2^{2p^e})(x_2 + b_2)^{(N_2 - 1)p^e} \in (a_2^{p^e}, I) $. 
Using that $a_2, x_2 + b_2, x_3, \ldots, x_d $ give a system of parameters for $R$, we see $(cx_2^{p^e} - r_2x_2^{2p^e}) \in (a_2^{p^e}, I) \subseteq (J^{[p^e]},I) $ and conclude $cx_2^{p^e} \in (J^{[p^e]},x_2^{2p^e}, I)$ as desired.

Statement (ii.) proceeds similarly.  We have that $a_d$ is a non-zero divisor on $R / (x_2, \ldots, x_{d-1})$ as its multiple $x_d^n x_1^n$ is a non-zero divisor, and again $(x_d, J)$ is not contained in any minimal prime of $(a_d, x_2, \ldots, x_{d-1})$.  We can find $b_d \in J$ so that $x_d + b_d$ is not in any minimal prime of $(a_d, x_2, \ldots, x_{d-1})$ and thus $a_d, x_2, \ldots, x_{d-1}, x_d + b_d$ is a system of parameters for $R$.  From $x_d^n J^{(n)} \subseteq a_d R$, we have $nD + \divisor(x_d^n) \geq \divisor(a_d)$ and it follows $x_d^{p^e}J^{(p^e)} \subseteq a_d^{\lfloor p^e / n \rfloor}R$.
To see the final inclusion, let $I = (x_2^{N_2 p^e},\ldots, x_{d-1}^{N_{d-1} p^e})$ and suppose $cx_d^{(N_d -1) p^e} \in (J^{(p^e)},I,x_d^{N_dp^e})$ for some $c \in R$.  We can find $r_d \in R$ so that $(c - r_dx_d^{p^e})x_d^{(N_d -1) p^e} \in (J^{(p^e)},I)$.  As $b_d \in J$ so $b_d^{p^e} \in J^{(p^e)}$, hence also $(c - r_dx_d^{p^e})(x_d + b_d)^{(N_d -1) p^e} \in (J^{(p^e)},I)$. Multiplying through by $x_d^{p^e}$ gives $ (cx_d^{p^e} - r_dx_d^{2p^e})(x_d + b_d)^{(N_d -1) p^e} \in (a_d^{\lfloor p^e/n \rfloor},I)$.  Using that $a_d, x_2, \ldots, x_{d-1}, x_d + b_d$ are a system of parameters, we see $ (cx_d^{p^e} - r_dx_d^{2p^e}) \in (a_d^{\lfloor p^e/n \rfloor},I)$.  Multiplying through by $x_1^n$ and using that $a_d^{\lfloor p^e/n \rfloor}x_1^n \in (J^{(n)})^ {\lfloor p^e / n \rfloor + 1} \subseteq J^{(p^e)}$ gives $ (cx_1^n x_d^{p^e} - r_dx_1^n x_d^{2p^e}) \in (J^{(p^e)},I)$ and in particular
$ cx_1^n x_d^{p^e} \in (J^{(p^e)},I, x_d^{2p^e})$ as desired.
\end{proof}

\begin{Corollary}
\label{QGorAnalysis}
Suppose that $(R,\m,k)$ is a complete Cohen-Macaulay local $F$-finite domain of dimension $d$.
\begin{enumerate}[(i.)]
\item
\cite{HunekeLeuschke}
If $R$ is Gorenstein and $x_1, \ldots, x_d$ are any system of parameters for $R$ and $\delta \in R$ generates the socle of $R / (x_1, \ldots, x_d)$, then 
$
I_e^{\mathrm{F-sig}} = ((x_1, \ldots, x_d)^{[p^{e}]}:\delta^{p^{e}})
$
for all $e \in \N$ and
$
 s(R) =  e_{HK}((x_1, \ldots, x_d)) - e_{HK}(( x_1, \ldots, x_d, \delta)).
$
\item
If the punctured spectrum $\Spec(R) \setminus \{\m\}$ is $\Q$-Gorenstein, there exists an $\m$-primary ideal $I$ and $x \in R$ with $(I:x) = \m$ so that $s(R) = e_{HK}(I) - e_{HK}((I,x))$ and moreover  $I_e^{\mathrm{F-sig}} \subseteq ((I^{[p^{e}]})^*: x^{p^{e}})$ for all $e \in \N$. In particular, if $R$ is weakly $F$-regular, $I_e^{\mathrm{F-sig}} = (I^{[p^{e}]}: x^{p^{e}})$ for all $e \in \N$.
\end{enumerate}
\end{Corollary}

\begin{proof}
(i.) Since $R$ is Gorenstein, we may take $J = (x_1)$ to be the canonical ideal above in Corollary~\ref{reducex1}.  For fixed $e \in \N$, we have that
$I_e^{\mathrm{F-sig}} = ((x_1^{tp^{e}}, \ldots, x_d^{tp^{e}}):(x_1^{(t-1)p^{e}} \cdots x_d^{(t-1)p^{e}} \delta^{p^e}))$ for all $t \gg 0$.  However, since $x_1, \ldots, x_d$ are a regular sequence on $R$, it follows that $((x_1^{tp^{e}}, \ldots, x_d^{tp^{e}}):(x_1^{(t-1)p^{e}} \cdots x_d^{(t-1)p^{e}} \delta^{p^e})) = ((x_1^{p^{e}}, \ldots, x_d^{p^{e}}): \delta^{p^e})$ for any $t \in \N$.  The final statement follows again using the relation in \eqref{relHKtocolon} once more.

For (ii.), we begin with the following construction.  Let $J$ be a canonical ideal of $R$, and choose $0 \neq x_1 \in J$.  Let $U_2$ be the complement of the minimal primes of $x_1$, which are all height one contain the set of minimal primes of $J$.   We have that $U_2^{-1}J$ is principal (as $U_2^{-1}R$ is a semi-local Dedekind domain) and can choose $x_2 \in U_2$ and $a_2 \in J$ so that $J_{x_2} = a_2R_{x_2}$ and $x_2 J \subseteq a_2 R$.


As $R_\p$ is $\Q$-Gorenstein for all $\p \in \Spec(R) \setminus \{\m \}$, we may choose $n \in \N$ so that $J^{(n)}_\p$ is principal for all $\p \in \Spec(R) \setminus \{\m \}$. Assuming $x_1, \ldots, x_{i-1}$ have already been chosen, let $U_i$ be the complement of the set of minimal primes of $x_1, \ldots, x_{i-1}$.  Again, $U_i^{-1}J^{(n)}$ is principal (as $U_i^{-1}R$ is a semilocal domain, every locally principal ideal is principal by \cite[Theorem 60]{KaplanskyCommutativeRings}) so we can choose $x_i \in U_i$ and $a_i \in J^{(n)}$ so that $J_{x_i}^{(n)} = a_i R_{x_i}$ and $x_i^n J^{(n)} \subseteq a_i R$. Continuing in this manner gives a system of parameters $x_1, \ldots, x_d$ of $R$ with $x_1 \in J$, and a sequence of elements $a_3, \ldots, a_d \in J^{(n)}$ so that $x_i^n J^{(n)} \subseteq a_i R$ for each $i = 3, \ldots, d$. For fixed $e \in \N$ and $t \gg 0$, we have from Corollary \ref{reducex1} and repeated application of Lemma \ref{colonslemma} (ii.) for $x_d, \ldots, x_3$ and then Lemma \ref{colonslemma} (i.) for $x_2$ that
\begin{equation}
\label{eq:mess}
\begin{array}{lll}
 I_e^{\mathrm{F-sig}} & = & ((J^{[p^{e}]},x_2^{t p^e},\ldots, x_{d-1}^{t p^e},  x_d^{t p^e}):(x_2^{(t - 1)p^e} \cdots x_{d-1}^{(t - 1)p^e} x_d^{(t - 1)p^e}\delta^{p^e})) \\ 
 & \subseteq & ((J^{(p^{e})},x_2^{t p^e}, \ldots, x_{d-1}^{t p^e}, x_d^{t p^e}):(x_2^{(t - 1)p^e} \cdots  x_{d-1}^{(t - 1)p^e}  x_d^{(t - 1)p^e}\delta^{p^e})) \\
 & \subseteq & ((J^{(p^{e})},x_2^{t p^e}, \ldots, x_{d-1}^{t p^e}, x_d^{2 p^e}):(x_1^n x_2^{(t - 1)p^e} \cdots  x_{d-1}^{(t - 1)p^e}  x_d^{p^e}\delta^{p^e})) \\
 &\subseteq & \qquad \qquad \qquad \vdots \qquad \qquad \qquad \vdots \qquad \qquad \qquad \vdots \\
 & \subseteq & ((J^{(p^{e})},x_2^{t p^e}, x_3^{2p^e}, \ldots, x_d^{2 p^e}):(x_1^{(d-2)n}x_2^{(t - 1)p^e}x_3^{p^e} \cdots  x_d^{p^e}\delta^{p^e})) \\ 
 & = & ((J^{[p^{e}]},x_2^{2 p^e}, x_3^{2p^e}, \ldots, x_d^{2 p^e}):(x_1^{(d-2)n}x_2^{p^e}x_3^{p^e} \cdots  x_d^{p^e}\delta^{p^e})) \\ 
  & = & \left(\left((J^{[p^{e}]},x_2^{2 p^e},  \ldots, x_d^{2 p^e}):(x_2^{p^e} \cdots  x_d^{p^e}\delta^{p^e})\right):x_1^{(d-2)n}\right)
\end{array}
\end{equation}
so setting $I = (J, x_2^2, \ldots, x_d^2)$ and $x = x_2 \cdots x_d \delta$ gives $s(R) = e_{HK}(I) - e_{HK}((I,x))$ using Lemma~\ref{colonsamelimit} (with $I_e = ((J^{[p^{e}]},x_2^{2 p^e},  \ldots, x_d^{2 p^e}):(x_2^{p^e} \cdots  x_d^{p^e}\delta^{p^e}))$, $J_e = I_e^{\mathrm{F-sig}}$, and $c = x_1^{(d-2)n}$).

For the final statement, since $(I_e^{\mathrm{F-sig}})^{[p^{e'}]} \subseteq I_{e+e'}^{\mathrm{F-sig}}$ for all $e,e' \in \N$, it follows from \eqref{eq:mess} that 
\[
0 \neq x_1^{(d-2)n} \in \bigcap_{e'\in \N} \left(I^{[p^{e+e'}]}:(x^{p^e}I_e^{\mathrm{F-sig}})^{[p^{e'}]} \right)
\]
so that $x^{p^e}I_e^{\mathrm{F-sig}} \subseteq (I^{[p^e]})^*$ or $I_e^{\mathrm{F-sig}} \subseteq ((I^{[p^e]})^*:x^{p^e})$ as claimed.
In particular, if $R$ is weakly $F$-regular, we have that $I_e^{\mathrm{F-sig}} \subseteq ((I^{[p^e]}):x^{p^e})$ and equality follows from Lemma~\ref{Iecontainscolons}.
\end{proof}



One can also use Theorem \ref{WY Type result} to show that the length criteria in Corollary~\ref{Length Criterion for at-tight closure} and Corollary~\ref{Length Criterion for Delta closure} also determine the $F$-signatures in those settings.

\begin{Corollary}\label{a^t signature and WY result} Let $(R,\m,k)$ be a complete $F$-finite local domain of dimension~$d$.  Let $\a \subseteq R$ be a non-zero ideal and $\xi \in \R_{\geq 0}$. Suppose $\{J_t\}_{t \in \N}$ is a descending chain of irreducible ideals cofinal with the powers of $\m$, and $\delta_t \in R$  generates the socle of $R/J_t$.  Then $s(R, \a^\xi) = \lim_{t \to \infty} \lim_{e \to \infty} \frac{1}{p^{ed}} \ell_R (R/(J_t^{[p^e]} : \a^{\lceil \xi(p^e-1)\rceil}\delta_t^{p^e}))$.\end{Corollary}

\begin{proof}
It is straightforward to check for each $e \in \N$ that
\begin{eqnarray*}
 I_e^{\a^\xi} &=& ( x \in R \mid \phi(x^{1/p^e}) \in \m \mbox{ for all } \phi \in \sC^{\a^\xi}_e ) \\
 & = & ( x \in R \mid (xy)^{1/p^e}R \subseteq R^{1/p^e} \mbox{ is not split for any } y \in \a^{\lceil \xi(p^e-1)\rceil} )
\end{eqnarray*}
 and it follows that $I_e^{\a^\xi} = \sum_{t \in \N} I_{t,e}^{\a^\xi}$ where $I_{t,e}^{\a^\xi} = ( J_t^{[p^e]} : \delta_t^{p^e} \a^{\lceil (p^e-1)\xi \rceil})$ for $t,e \in \N$. Choosing an element $0 \neq c \in R$ such that $c \a^{\lceil \xi (p^{e+1} - 1) \rceil} \subseteq (\a^{\lceil \xi(p^e - 1)\rceil})^{[p]}$ for all $e \in \N$ as in the proof of Theorem~\ref{ateefsigexists}, we have that $\m^{[p^e]} \subseteq I_{t,e}^{\a^\xi}$, $c((I_{t,e}^{\a^\xi})^{[p]}) \subseteq I_{t,e+1}^{\a^\xi}$, and $I_{t,e}^{\a^\xi} \subseteq I_{t+1,e}^{\a^\xi}$ for all $t,e \in \N$.  The result now follows immediately from Theorem~\ref{WY Type result} as we know $s(R, \a^\xi) = \lim_{e \to \infty} \frac{1}{p^{ed}} \ell_R (R/I_e^{\a^\xi})$.
\end{proof}

\begin{Corollary}\label{Delta signature and WY result} Let $(R,\m,k)$ be a complete $F$-finite local normal domain of dimension~$d$.  Suppose $\{J_t\}_{t \in \N}$ is a descending chain of irreducible ideals cofinal with the powers of $\m$, and $\delta_t \in R$  generates the socle of $R/J_t$.  If  $\Delta$ is an effective $\Q$-divisor on $\Spec(R)$, then $s(R, \Delta) = \lim_{t \to \infty} \lim_{e \to \infty} \frac{1}{p^{ed}} \ell_R (R/(J_t^{[p^e]}R(\lceil (p^e - 1) \Delta \rceil) : \delta_t^{p^e}))$.
\end{Corollary}

\begin{proof}
Again it is straightforward to check for each $e \in \N$ that
\begin{eqnarray*}
 I_e^{(R,\Delta)} &=& ( x \in R \mid \phi(x^{1/p^e}) \in \m \mbox{ for all } \phi \in \sC^{(R,\Delta)}_e ) \\
 & = & ( x \in R \mid (x)^{1/p^e}R \subseteq (R(\lceil (p^e - 1)\Delta \rceil))^{1/p^e} \mbox{ is not split} )
\end{eqnarray*}
 and it follows that $I_e^{(R,\Delta)} = \sum_{t \in \N} I_{t,e}^{(R,\Delta)}$ where $I_{t,e}^{(R,\Delta)} = ( J_t^{[p^e]} R(\lceil (p^e - 1)\Delta \rceil) : \delta_t^{p^e} )$ for $t,e \in \N$. 
 Choosing an element $0 \neq c \in R$ 
 so that $\mathrm{div}_R(c) \geq p\lceil \Delta \rceil$ as in the proof of Theorem~\ref{deltafsigexists}, we have that we have that $\m^{[p^e]} \subseteq I_{t,e}^{(R,\Delta)}$, $c((I_{t,e}^{(R,\Delta)})^{[p]}) \subseteq I_{t,e+1}^{(R,\Delta)}$, and $I_{t,e}^{(R,\Delta)} \subseteq I_{t+1,e}^{\a^\xi}$ for all $t,e \in \N$.  The result now follows immediately from Theorem~\ref{WY Type result} as we know $s(R, \Delta) = \lim_{e \to \infty} \frac{1}{p^{ed}} \ell_R (R/I_e^{(R,\Delta)})$.
\end{proof}

Lastly, note that Theorem~\ref{WY Type result} can also be applied outside the context of descending chains of irreducible ideals, and gives yet another perspective on the original proof of the existence of the $F$-signature.

\begin{Corollary}
\cite[Proof of Theorem~4.9]{Tucker2012}
 Let $(R, \m, k)$ be a complete $F$-finite local domain of dimension $d$ and  $I_e^{\mathrm{F-sig}} = ( r \in R \mid \phi(r^{1/p^e}) \in \m \mbox{ for all } \phi \in \Hom_R(R^{1/p^e},R) )$.  Then $s(R) = \lim_{t \to \infty} \frac{1}{p^{td}} e_{HK}(I_t^{\mathrm{F-sig}})$.
\end{Corollary}

\begin{proof}
 Let $I_{t,e} = \left\{
\begin{array}{c@{\qquad}l}
 (I_t^{\mathrm{F-sig}})^{[p^{e-t}]} & t < e \\
 I_e^{\mathrm{F-sig}} & t \geq e
\end{array}
  \right.$. It is easily checked that $\m^{[p^e]} \subseteq I_{t,e}$, $I_{t,e}^{[p]} \subseteq I_{t, e+1}$, and $I_{t,e} \subseteq I_{t+1,e}$ for all $t,e \in \N$ using that $\m^{[p^e]} \subseteq I_e^{\mathrm{F-sig}}$ and $(I_e^{\mathrm{F-sig}})^{[p]} \subseteq I_{e+1}^{\mathrm{F-sig}}$ for all $e \in \N$.  The result now follows immediately from Theorem~\ref{WY Type result}, as we have
  \begin{eqnarray*}
 s(R) &=& \lim_{e \to \infty} \frac{1}{p^{ed}} \ell_R(I_e^{\mathrm{F-sig}}) =  \lim_{e\rightarrow \infty}\lim_{t\rightarrow \infty}\frac{1}{p^{ed}}\ell_R(R/I_{t,e})  =
 \lim_{t\rightarrow \infty}\lim_{e\rightarrow \infty}\frac{1}{p^{ed}}\ell_R(R/I_{t,e}) \\ &=& \lim_{t\rightarrow \infty}\lim_{e\rightarrow \infty}\frac{1}{p^{ed}}\ell_R(R/(I_t^{\mathrm{F-sig}})^{[p^{e-t}]}) = \lim_{t\rightarrow \infty} \frac{1}{p^{td}} e_{HK}(I_t^{\mathrm{F-sig}})
   \end{eqnarray*}
 as desired.
\end{proof}

\section{Open Questions}\label{Section Open Questions}

In this section, we collect together some information on the important questions left unanswered in this article.  We have seen that Hilbert-Kunz multiplicity and the F-signature enjoy semi-continuity properties for $F$-finite rings; it is not difficult to extend these results to rings which are essentially of finite type over an excellent local ring. More generally, however, this raises the following question.

\begin{Question}
If $R$ is an excellent domain that is not $F$-finite or essentially of finite type over an excellent local ring, do the Hilbert-Kunz multiplicity and $F$-signature determine semicontinuous $\R$-valued functions on $\Spec(R)$?
\end{Question}

\noindent
In the case of $F$-signature, a positive answer would imply the openness of the strongly $F$-regular locus for such a ring.  Note that this question would seem closely related to the existence of (locally and completely stable) test elements for excellent domains, which also remains unanswered.

Perhaps the most important question left open in this article regarding the relationship between Hilbert-Kunz multiplicity and $F$-signature is the following, which (in light of the results of the previous section) is  attributed to Watanabe and Yoshida in \cite{WY}.

\begin{Question}\label{Weak implies Strong} Let $(R,\m,k)$ be a complete local $F$-finite normal domain. Do there exist $\m$-primary ideals $I\subsetneq J$ so that $\e(I)-\e(J)=s(R)$?
\end{Question}

\noindent
For any ring such that Question~\ref{Weak implies Strong} has a positive answer, it follows from weak $F$-regularity is equivalent to strong $F$-regularity by the length criterion for tight closure \cite[Theorem~8.17]{HHTams} (Corollary~\ref{Length Criterion for Tight Closure} above).  

We see from Corollary~\ref{QGorAnalysis} that  Question~\ref{Weak implies Strong} is true provided $R$ is $\Q$-Gorenstein on the punctured spectrum. When in addition $R$ is weakly $F$-regular, the stronger condition below is satisfied as well.



\begin{Question} 
\label{relativeHKfunctiongivessplittingnumbers}
Let $(R,\m,k)$ be a complete local $F$-finite normal domain. Do there exist $\m$-primary ideals $I\subsetneq J$ so that 
\[
\frac{\frk_R(R^{1/p^e})}{[k^{1/p^e}:k]} = \frac{\ell_R(R/I^{[p^e]}) - \ell_R(R/J^{[p^e]})}{\ell_R(J / I)}
\]
for all $e \in \N$?
\end{Question}

\noindent
The importance of Question \ref{relativeHKfunctiongivessplittingnumbers} stems from the observation that it allows one to apply the results of Huneke, McDermott, and Monsky \cite{SecondCoefficient}; for an $\m$-primary ideal $I \subseteq R$, they show the existence of a constant $\alpha(I)\in \R$ so that $\ell(R/I^{[p^e]})=\e(I)p^{ed}+\alpha(I)p^{e(d-1)}+O(p^{e(d-2)})$. In other words, Hilbert-Kunz functions in normal local $F$-finite domains are polynomial in $p^e$ to an extra degree.  In particular, for any ring such that Question~\ref{relativeHKfunctiongivessplittingnumbers} has a positive answer, so also does the following question.

\begin{Question}
\label{F-sig second coefficient}
 If $(R,\m,k)$ is a complete local $F$-finite normal domain, when does there exist a positive constant $\alpha(R) \in \R$ so that
$\displaystyle
 \frac{\frk_R(R^{1/p^e})}{[k^{1/p^e}:k]} = s(R) p^{ed} + \alpha(R) p^{e(d-1)} + O(p^{e(d-2)})
$?
\end{Question}

Finally, we have tried to emphasize the applicability of our techniques to the settings of divisor and ideal pairs throughout, and the questions above are readily generalized to those settings.  Moreover, it may well be the case that answers to the questions above require the use of such pairs (particularly to remove Gorenstein or $\Q$-Gorensetein hypotheses).  In this direction, and in view of \cite{SchwedeNonQGorTestIdeals}, one could imagine a positive and constructive answer to the question below could prove quite useful.

\begin{Question} 
Suppose $(R, \m, k)$ is an $F$-finite local normal domain and $X = \Spec(R)$.  For each $\epsilon > 0$, does there exist an effective $\Q$-divisor $\Delta$ on $X$ such that $K_X + \Delta$ is $\Q$-Cartier with index prime to $p$ and so that $s(R) - s(R, \Delta) < \epsilon$?  In case $R$ is not local, can $\Delta$ be chosen globally to have this property locally over $\Spec(R)$?
\end{Question}

\bibliographystyle{amsalpha}
\bibliography{References}

\end{document}